\documentclass{article}
\usepackage{amsthm,bbm,amsmath,MnSymbol}
	\usepackage{todonotes}
	\usepackage{mathtools}
	\usepackage{mathrsfs}
	\usepackage[all]{xy}
	\usepackage{nicefrac}
	\usepackage{esvect}
	\usepackage[english]{babel}

\usepackage{color} 
\usepackage{hyperref}

\usepackage[T1]{fontenc}
\usepackage{bbm,MnSymbol}
\usepackage{mathtools}

\usepackage{amsfonts}

\usepackage[utf8]{inputenc}

\newtheorem{theorem}{Theorem}[section]
\newtheorem{definition}[theorem]{Definition}
\newtheorem{lemma}[theorem]{Lemma}
\newtheorem{remark}[theorem]{Remark}
\newtheorem{conjecture}[theorem]{Conjecture}
\newtheorem{proposition}[theorem]{Proposition}
\newtheorem{corollary}[theorem]{Corollary}

\newtheorem{example}[theorem]{Example}

\def\T{\mathcal{T}}
\def\L{\mathcal{L}}
\def\G{\mathcal{G}}
\def\P{\mathcal{P}}
\def\X{\mathcal{X}}

\def\B{\mathcal{B}}
\def\W{\mathcal{W}}
\def\mathbb{\mathbbm}

\def\Y{\mathcal{Y}}
\def\V{\mathcal{V}}
\def\U{\mathcal{U}}
\def\F{\mathcal{F}}

\def\cite#1{#1}
\def\uS{\mathcal{S}}

\def\p{\mathfrak{p}}

\def\K{\mathcal{K}}
\def\ref#1{#1}
\def\Tor{\rm{Tor}}

\def\Inv.lt.{\rm{Inv.\, lt.}}
\def\dir.lt.{\rm{dir.\, lt.}}
\def\top{\rm{top}}
\def\cont{\rm{cont}}
\def\Hom{\rm{Hom}}
\def\Image{\rm{Image}}
\def\Tor{\mathcal{T}or}
\def\Id{\rm{Id}}
\def\aff{\rm{aff}}
\def\SL{\rm{SL}}
\def\D{\mathfrak{D}}
\def\IG{\rm{IG}}

\def\loc{\rm{loc}}
	\def\xunderline#1{\ensuremath{\textit{#1}}}

\begin{document}

\title{Conjectural Positivity for  Pontryagin Product in Equivariant K-theory of Loop Groups}

\author{Shrawan Kumar}

\date{}

\maketitle{}

\begin{center} {\it To my friend Daniel Nakano on the occasion of his sixtieth birthday}
\end{center}
\vskip2ex

{\it Abstract:} {\tiny Let $G$ be a connected simply-connected simple algebraic group over $\mathbb{C}$ and let $T$ be  a maximal torus, $B\supset T$ a Borel subgroup   and $K$  a maximal compact subgroup. Then, the product in the (algebraic) based loop group $\Omega(K)$  gives rise to a comultiplication in the topological $T$-equivariant $K$-ring $K_T^{\top}(\Omega(K))$. Recall that   $\Omega(K)$ is identified with the affine Grassmannian $\X$ (of $G$) and hence we get a comultiplication in 
$ K_T^{\top}(\X)$. Dualizing, one gets the Pontryagin product in the $T$-equivariant $K$-homology $K^T_0(\X)$, which in-turn gets identified with the convolution product (due to S. Kato).  
Now, $ K_T^{\top}(\X)$ has a basis $\{\xi^w\}$ over the representation ring $R(T)$  
given by the ideal sheaves corresponding to the finite codimension Schubert varieties $X^w$ in $\X$. We make a positivity conjecture  on the comultiplication structure constants in the above basis. Using some results of  Kato, this conjecture gives rise to an equivalent conjecture on the positivity of the multiplicative structure constants in $T$-equivariant quantum $K$-theory $QK_T(G/B)$ in the Schubert basis.}
\setcounter{section}{0}

\section{Introduction}

Let $G$ be a connected simply-connected simple algebraic group over $\mathbb{C}$. We fix a Borel subgroup $B$ and  a maximal torus ${T}\subset {B}$. We also fix a maximal compact subgroup $K$ of $G$ such that $T_o:=T\cap K$ is a (compact) maximal torus of $K$. 
Let $\X = G(\mathbb{C}((t))/G(\mathbb{C}[[t]])$  be the  affine Grassmannian. Then, $\X$ is an ind-projective variety with filtration 
$$\X_0\subset \X_1\subset \ldots \subset\X_n \subset\ldots \,\,\text{ given by Schubert varieties.}
$$ Let $
K_T^{\top}\left(\X\right)= \underset{n\rightarrow \infty}{\Inv.lt.}\,\,K_T^{\top}\left(\X_n\right)$ be the topological $T$-equivariant  $K$-group of $\X$  under the analytic topology on $\X_n$. Let $\W:=W\propto Q^\vee$ be the affine Weyl group, where $Q^\vee$ is the coroot lattice of $G$ and  $W$ is the (finite) Weyl group of $G$ and let  $\W'$ be the set of minimal  coset representatives in $\W/W$. Let $\underline{o}$ be the base point of ${\X}$ and $\U^-:= G[t^{-1}]$. 
For any $ w\in \W'$, the sheaf $\xi^w:=\mathcal{O}_{ X^w} \left( -\partial  X^w \right)$ over ${\X}$ gives rise to an element denoted $[\bar\xi^w]\in 
K_T^{\top}\left(\X\right)$ by using Lemma \ref{2.2}, where $X^w := \overline{\U^- w \underline{o}}\subset \X$ and
$\partial  X^w :=  X^w \backslash \U^- w \underline{o}$.  By Lemma \ref{2.7}, 
$$
K_T^{\top} \left(\X\right) = \Pi_{w\in \W'} \,R(T) \left[\bar{\xi}^{w}\right].
$$
We also define the  $T$-equivariant $K$-homology $K^T_0 \left( {\X}\right)$ by 
$
K^T_0 \left( {\X}\right) = \underset{n\rightarrow \infty}{\dir.lt.}\,\, K^T_0 \left( {\X}_n\right), 
$ where $K^T_0 \left( {\X}_n\right)$ is the Grothendieck group corresponding to the $T$-equivariant coherent sheaves on the projective variety $\X_n$. Then, as in Definition \ref{2.6},
$$
K_0^T\left(\X\right) = \bigoplus_{ w\in \W'} R(T)\cdot \left[\mathcal{O}_{X_w}\right],
$$
where $X_w$ is the Schubert variety $\overline{\B w \underline{o}}$, $\B$ being the standard Iwahori subgroup defined as  the inverse image of $B$ in $G[[t]]$ under the evaluation map at $t=0$.

Let $\Omega (K)$ be the based {\it algebraic}  loop group of $K$ endowed with the analytic topology (see the details above Lemma \ref{2.5}). 
Then, $K$ (in particular, $T_0$) acts on $\Omega (K)$ via conjugation. We recall the following well-known lemma (cf. Lemma \ref{2.5}):
\begin{lemma} 
The inclusion map 
$$ 
\beta: \Omega (K) \rightarrow \X,~ \gamma \mapsto {\gamma}\cdot \underline{o},\quad \mbox{for } \gamma\in \Omega (K)
$$
is a $K$-equivariant homeomorphism  under the analytic topology on $\X$.  
\end{lemma}

Consider the $K$-equivariant multiplication map (which is continuous) 
$$
\hat{m}:~\Omega (K) \times \Omega (K) \rightarrow \Omega (K), \quad \left( \gamma_1,  \gamma_2\right) \mapsto \gamma_1\cdot \gamma_2.
$$
By virtue of the above $K$-equivariant homeomorphism $\beta$, we get a $K$-equivariant {\it continuous} (but {\it not} regular) map 
$
m:~\X\times .\X \rightarrow \X.
$
Thus, we get a pull-back map to the completed tensor product (cf. Definition \ref{2.6}):
$$
m^*:~K^{\top}_{T_0}\left(\X\right) = K^{\top}_T \left(\X\right) \rightarrow K^{\top}_T \left(\X\times \X\right) = K_T^{\top} \left(\X\right) {\hat{\otimes}}_{R(T)} K_T^{\top} \left(\X\right).
$$
The induced map $m^*$
can be written as follows (for any $w\in\W'$):
$$
m^*\left(\left[\bar{\xi}^w\right]\right) = \sum_{u,v\in \W'} a^w _{u,v}\left[\bar{\xi}^u\right] \otimes\left [\bar{\xi}^v\right],\mbox{ for unique } a^w_{u,v}\in R(T).
$$
The following is our main conjecture (cf. Conjecture \ref{2.9}). 

\begin{conjecture}\label{1.2}
We conjecture that for any $u,v,w\in\W'$, 
$$
(-1)^{\ell(u)+\ell(v)-\ell(w)} a^w_{u,v}\in \mathbb{Z}_+\left[\left(e^{\alpha_1}-1\right), \ldots , \left(e^{\alpha_l}-1\right)\right],
$$
 i.e., $(-1)^{\ell(u)+\ell(v)-\ell(w)} a^w_{u,v}$ is a polynomial in the variables $e^{\alpha_1}-1,\ldots, e^{\alpha_l}-1$ with non-negative integral coefficients, where $\left\{\alpha_1,\ldots,\alpha_l\right\}$ are the simple roots of $G$. 
\end{conjecture}
The pairing 
$$
\langle ~,~\rangle: ~ K^0_T\left(\bar{\X}\right) \otimes_{R(T)}  K_0^T\left({\X}\right) \rightarrow R(T)
$$
as in Definition \ref{2.3} is non-singular by Theorem \ref{2.4}, where $K^0_T\left(\bar{\X}\right)$ is defined above Definition \ref{2.3}. Further, 
$ K^0_T\left(\bar{\X}\right)$ is canonically isomorphic with $ 
 K^{\top}_T \left(\X\right)$ (cf. Lemma \ref{2.7}).
Thus, the pairing  induces the identification: 
$$
\psi :  K^{\top}_T\left({\X}\right) \simeq \Hom_{R(T)} \left(K_0^T\left({\X}\right) ,R(T)\right),
$$
and a similar identification $\tilde{\psi}$ for $\X\times \X$. 
Using these identifications $\psi$ and $\tilde{\psi}$, we can rewrite the map ${m}^*$ as 
$$
\hat{m}^*:\Hom_{R(T)} \left(K_0^T\left({\X}\right) ,R(T)\right) \rightarrow \Hom_{R(T)} \left(K_0^{T}\left(\X\times \X\right) ,R(T)\right)
$$
giving rise to the product 
$$
{\p}:~K_0^{T}\left(\X\times \X\right) \simeq K^T_0\left({\X}\right) \otimes_{R(T)} K_0^T\left({\X}\right) \rightarrow K_0^T\left({\X}\right).
$$
Thus, $\p$ makes $K_0^T\left({\X}\right)$ into an $R(T)$-algebra (cf. Definition \ref{3.1} for more details). Its product is called the {\it Pontryagin product}. Let us write,  under the Pontryagin product, 
for $u,v \in \W'$,
\begin{equation*}
\left[\mathcal{O}_{X_u}\right]* \left[\mathcal{O}_{X_v}\right]= \sum_{ w\in \W'}b_{u,v}^w\left [\mathcal{O}_{X_w}\right]. 
\end{equation*}
By  Lemma \ref{3.2}, using
Theorem \ref{2.4}, we get,  for any $u,v,w\in\W'$, 
$$
a_{u,v}^w = b_{u,v}^w.
$$
Thus, the above Conjecture \ref{1.2}  translates to the following equivalent conjecture on the Pontryagin product in $K_0^T\left({\X}\right)$ (cf. Conjecture \ref{3.3}). 
\begin{conjecture} \label{1.3} Under the Pontryagin product as above, its structure constants $b_{u,v}^w$ satisfy 
$$
(-1)^{ \ell(u)+\ell(v)-\ell(w)}
\,\,b_{u,v}^w \in \mathbbm{Z}_+ \left[\left(e^{\alpha_1}-1\right), \ldots , \left(e^{\alpha_l}-1\right)\right].
$$
\end{conjecture}
Consider the diagram
$$
\begin{array}{lll}
\tilde{\X}:=& \G \times^{\B} \X & \xrightarrow{\mu} \X,\\
&\quad \downarrow \pi&\\
& \,\,\, \Y
\end{array}
$$
where $\Y:=\G/\B$, $\mu([g,x]):=g\cdot x$ and $\pi ([g,x]):= g\B$ for $g\in \G$ and $x\in \X$. 
Take $\B$-equivariant coherent sheaves $\mathcal{S}_1$ on $\Y$ and $\mathcal{S}_2$ on $\X$ supported in $p^{-1}(\X_n)$ and $\X_n$  respectively (for some $n>0$), where $p:\Y \rightarrow \X$ is the projection. Their {\it convolution product} is defined by 
$$
\mathcal{S}_1 \odot' \mathcal{S}_2 := \mu_! \left( \left( \pi^* \mathcal{S}_1\right) \otimes^L \left( \epsilon \boxtimes^{\B}\mathcal{S}_2\right)\right) \in K_0^\B \left(\X\right),
$$
where $\epsilon \boxtimes^{\B}\mathcal{S}_2$ denotes the sheaf on $\G\times^{\B} \X$ the pull-back of which to $\G \times \X$ is the product sheaf $\epsilon \boxtimes\mathcal{S}_2$ 
($\epsilon$ being the rank-1 trivial bundle over $\G$),
 $\otimes^L$ is the derived tensor product $\sum(-1)^i \Tor_i^{\mathcal{O}_{\tilde{\X}}}$ and $\mu_! :=\sum_i (-1)^i R^i \mu_*$.
Since $\mu_!$ and $\otimes^L$ both descend to corresponding $K$-groups, we get a well defined map
$$
\odot ': ~ K_0^\B(\Y) \otimes _\mathbbm{Z} K_0^\B \left(\X\right) \rightarrow \K_0^\B \left(\X\right). 
$$
Observe that $\odot '$ is $R(\B)$-linear in the first variable but, in general, {\it not} $R(\B)$-linear in the second variable but it is $R(\P)$-linear (cf. Corollary \ref{4.5}). 
Then, we have the following result (cf. Proposition \ref{4.3}).

\begin{proposition} \label{1.4}
For $u\in \W$ and $v\in \W'$,
$$
\left[\mathcal{O}_{X^\B_u}\right] \odot'\left [\mathcal{O}_{X_v}\right] =\left [\mathcal{O}_{X_{\overline{u*v}}}\right] \in K_0^B\left(\X\right) ,
$$
where ${X^\B_u}:= \overline{\B u\B/\B} \subset \Y$ and $*$ is the Demazure product in $\W$ (cf. Definition \ref{4.2}).

Observe that $u*v$ may not lie in $\W'$. We take its unique representative $\overline{u*v}$ in $\W'$. 
\end{proposition}

Let $\left\{\omega_i\right\}_{1\leq i\leq l}$ be the fundamental weights of $G$. 
Recall that there is an isomorphism  $R(B)\xrightarrow{\sim} K^0_G(X)= K_0^G (X)$ explicitly given by
$e^\lambda \mapsto \left[\L\left(-\lambda\right)\right]$,  for a character $e^\lambda$ of $T$ ,
 where $\L\left(-\lambda\right)$ is the homogeneous line bundle over $X=G/B$ associated to the principal $B$-bundle $G\rightarrow X$ via the character $e^\lambda$ (cf. Definition \ref{4.6}). 
 As proved by Steinberg, $R(T)$ is a free $R(G)= R(T)^W$-module (under multiplication) with a basis
$
\left\{ e^{\delta_x} := x^{-1} \Pi_{\alpha_i: x^{-1} \alpha_i <0} \,\,e^{\omega_i}\right\}_{x\in W}.
$
Thus, $\left\{\L\left(-\delta_x\right)\right\}_{x\in W}$ is a basis of $K_0^G (X)$
as a $K_0^G (*)\simeq R(T) ^W$-module. 
Consider the pairing 
$$
\langle~,~\rangle :~ K_G^0 \left(X\right) \otimes_{K_G^0 \left(*\right) } K_G^0 \left(X\right)  \rightarrow K_G^0 \left(*\right) \simeq R(T)^W, 
\langle V_1,~V_2\rangle= \chi_G\left(V_1\otimes V_2\right),
$$
where $\chi_G$ denotes the $G$-equivariant Euler-Poincar\'e characteristic. 
Then, it is non-singular (cf. Derfinition \ref{4.6}).
Let $\left\{\L_x:=\L\left(-\delta_x\right)\right\}_{x\in W}$ be the Steinberg basis of $K_0^\P\left(\P/\B\right)\simeq K_\P^0\left(\P/\B\right)$ (since $\P/\B \simeq X$ is smooth) over $K_0^\P (*)$ and 
let $\left\{\L^x\right\}_{x\in W}$ be the dual basis of $K_0^\P\left(\P/\B\right)$ under the above pairing.

Let $\overline{\mu}:\P\times^\B \X\rightarrow \X$ be the product map $[p,x]\mapsto p\cdot x$, for $p\in\P$ and $x\in \X$. 
As mentioned earlier, $\odot'$ is \textit{not} $R(\B)$-linear in the second variable. To remedy this,  following S. Kato,
define the \textit{modified convolution product}: 
$
\odot :K_0^\B \left( \Y\right) \bigotimes _{K_0^\B(*)} K_0^\B \left( \X\right) \rightarrow K_0^\B \left ( \X\right)
$
by 
$$
a\odot b := \sum_{x\in W} \left(\epsilon \left(\L^x\right) \cdot a\right) \, \odot' \overline{\mu}_! \left(\L_x \overset{\B}{\boxtimes} b\right), \mbox{ for  } a\in K_0^\B\left(\Y\right) \mbox{ and } b \in K_0^\B\left(\X\right),
$$
where $\epsilon:K_0^\P \left(\P/\B\right)\overset{\sim}{\rightarrow}K_0^\B (*)$ is the isomorphism.
The following result is due to S. Kato (cf. Theorem \ref{4.9} and Corollary \ref{4.11}). 
\begin{theorem}\label{1.5}
 The two products $*$ and $\odot$ in $K_0^T(\X)$ coincide.  Moreover, the product $\odot$ in $K_0^T(\X)$ is associative and commutative. 
For any $u,v\in \W'$, write 
$$
\left[\mathcal{O}_{X_u}\right] \odot \left[\mathcal{O}_{X_v}\right] = \sum_{w\in \W'} p_{u,v}^w \left[ \mathcal{O}_{X_w}\right].
$$
Thus,
$$
p_{u,v}^w = b^w_{u,v}, \mbox{ for any } u,v,w\in \W',
$$
where $b_{u,v}^w$ are the structure constants for the Pontryagin product in $K_0^T(\X)$ as above. 
\end{theorem}
Thus, the above conjecture  can equivalently be reformulated in terms of the structure constants for the modified convolution product $\odot$ in $K_0^T(\X)$ (cf. Conjecture \ref{4.10}).  
\begin{conjecture}\label{1.6}
For any $u,v,w\in \W'$ ,
$$
(-1)^{\ell(u)+\ell(v)-\ell(w)}p_{u,v}^w \in \mathbb{Z}_+ \left[\left (e^{\alpha_{1}}-1\right),\ldots , \left (e^{\alpha_{l}}-1\right)\right].
$$
\end{conjecture}

For any $x\in W$, similar to the sheaf $\xi^w\left(w\in \W'\right)$, define the sheaf 
$$
\zeta^x =\mathcal{O} _{\mathring{X}^x}\left(-\partial \mathring{X} ^x\right),
$$
where $\mathring{X}^x := \overline{B^-xB/B} \subset X:= G/B$ ,
$
\partial \mathring{X}^x = \mathring{X}^x \backslash \left( B^-xB/B\right)$ and $B^-\supset T$ is the opposite Borel subgroup of $G$.
Consider its class $\left[ \zeta^x\right]\in K_0^T (X) = K^0_T(X)$.

Recall the $K_T^0(*)$-algebra isomorphism
$$
\varphi: ~ R(T) \underset{R(G)}{\otimes} R(T) \underset{\sim}{\rightarrow} K^0_T (X),~e^\lambda\otimes e^\mu \mapsto e^\lambda \cdot \L_X\left(-\mu\right). 
$$
The domain of $\varphi$ acquires the $K_T^0(*)=R(T)$-module structure via its multiplication on the first factor. 
The isomorphism $\varphi$ allows us to view $\zeta^x$ as an element $\bar{\zeta}^x\in R(T) \underset{R(G)}{\otimes} R(T) $. 
For any element $\alpha=\sum_j a_j \otimes b_j \in R(T) \underset{R(G)}{\otimes} R(T)$, 
we define 
$
\vert \alpha \vert = \sum_j a_j b_j \in R(T).
$
For any $0\leq i \leq l$, define a certain {\it left Demazure operator: }
$$
D_i' : R(T) \otimes_{R(G)} R(T) \rightarrow R(T)\otimes_{R(G)} R(T),
\,\,\,
D_i'\left( a\otimes b\right)=\left(D_i a\right)\otimes b, \,\text{for $a,b\in R(T)$}, 
$$
where, for any $0\leq i\leq l$, 
$
D_i:R(T) \to R(T) \,\,\text{takes $e^\lambda$ to  $\frac{e^\lambda-e^{s_i\lambda}}{1-e^{\alpha_i}}$}.
$ (Here $s_0:=s_\theta$ and $\alpha_0=-\theta$; $\theta$ being the highest root of $G$.)

The following is one of our main results (cf. Theorem \ref{5.9}). 
\begin{theorem}\label{1.7}
Take $u\in \W$, $v\in \W'$ and take a reduced decomposition $u=s_{i_1}\ldots s_{i_n}$\,  $\left(0\leq i_j\leq l\right)$. Then, under the modified convolution product
$$
\left[ \mathcal{O}_{X^\B_u}\right] \odot \left[ \mathcal{O}_{X_v}\right] = \sum_{x\in W} \sum_{1\leq j_1< \cdots <j_p\leq n} \left\vert D_{i_1}' \cdots \hat{\hat{D}}'_{i_{j_1}}\cdots \hat{\hat{D}}'_{i_{j_p}} \cdots D_{i_n}'\left(\bar{\zeta}^x \right)\right\vert [ \mathcal{O}_{X_{\overline{s_{i_{j_1}}*\cdots *s_{i_{j_p}}*x*v}}}],
$$
where $\hat{\hat{D}}'_j$ means to replace the operator $D_j'$ by the Weyl
group action on $R(T) \otimes_{R(G)} R(T)$ acting only on the first factor, $*$ is the Demazure product in $\W$  and for $w\in \W$, $\bar{w}$ denotes the corresponding minimal representative in $wW$. 
\end{theorem}
Define an involution 
$$
t:~R(T)\otimes_{R(G)}R(T) \rightarrow R(T)\otimes_{R(G)}R(T) ,~a\otimes b \mapsto b\otimes a,\,\mbox{for } a,b\in R(T).
$$
Via the isomorphism $\varphi$ identify any element of $K_T^0(X)$ by an element of $R(T)\otimes_{R(G)}R(T)$. 
Thus, for any class $\eta\in K_T^0(X)$, we have the transposed class $\eta^t := t(\eta) \in K_T^0 (X)$. The same definition as that of $\varphi$ realizes $\eta^t\in  K_T^0\left(\X^{\B}\right)$ compatible with its restriction to $X\hookrightarrow \X^\B$. Viewed $\eta^t$ as an element of $K_T^0\left(\X^{\B}\right)$, we write it as $\eta^t_{\aff}$. 
For any $u\in \W, v\in \W'$ and $x\in W$, consider 
$ X_{(u,x,v)} := X_u' \times^{\B}\mathring{X}_x' \times{^\B} X_v $
 together with the standard product map $\mu_x:X_{(u,x,v)}\rightarrow\X$ and the standard projection $\pi_x: X_{(u,x,v)}\rightarrow X_{u}^\B$, where $X'_u$ is the inverse image of $X_u^\B$ in $\G$ under $\G\rightarrow \Y$ and 
$\mathring{X}_x \subset X\hookrightarrow \Y.$ Here, $\mathring{X}_x$ is the Schubert variety $\overline{BxB/B} \subset X$ and $\mathring{X}_x'$ is to be thought of as its inverse image in $\G$. 
We have the standard pull-back map
$
\mu_x^*:K_T^0\left(\X\right) \rightarrow K_T^0\left( X_{(u,x,v)}\right).
$
We give another expression for the modified convolution product $\odot$ in the following (cf. Theorem \ref{5.15}):
\begin{theorem}\label{1.8}
For $u\in\W$ and $v\in\W'$, 
$$
\left[\mathcal{O}_{X_u^\B}\right]\odot\left[\mathcal{O}_{X_v}\right] =\sum_{w\in \W'}\,\,\sum_{x\in W} \left\langle \left( \left[ \zeta^x\right]^t_{\aff} \right)_{|X_u^\B},\pi_{x!}\mu_x^*\xi^w\right\rangle\left[\mathcal{O}_{X_w}\right].
$$
\end{theorem}

Using Theorem \ref{1.7}, we give an explicit expression in Section 6 for the convolution product $\odot$ in the affine Grassmannian associated to $G=\SL_2(\mathbb{C})$ (cf. Proposition \ref{6.3}). It was  obtained earlier in [LLMS] and also [Ka-1] by different methods. 

Let $Q_+^\vee:= \underset{i=1}{\overset{l}{\oplus}} \mathbbm{Z}_{\geq0}\alpha_i^\vee$, where $\left\{\alpha_1^\vee,\ldots,\alpha_l^\vee\right\}$ are the simple coroots of $G$. Consider the formal power series ring $\mathbbm{Z} [[Q_+^\vee]]$ in the variables $q_i=q^{\alpha_i^\vee}$. For any $\beta=\sum_{i=1}^ln_i\alpha_i^\vee$, $n_i\geq 0$ , we denote $q^\beta=
\prod q_i^{n_i}$. 
Additively, \textit{T-equivariant quantum K-theory} of $X=G/B$ is defined as 
$$
Q K_T \left(X\right) = K_T^0 \left(X\right) [[ q_1, \ldots , q_l]].
$$

Thus, $QK_T(X)$ has a $K_T^0(*) [[q_1,\ldots,q_l]] $-basis given by the structure sheaves
$\{ \left[ \mathcal{O}^x\right] = [ \mathcal{O}_{\mathring{X}_{xw_o}}]\}_{x\in W}$.  
It acquires a ring structure given by Givental  and Lee. We denote the product structure by $*$ called the \textit{quantum product}.
Then, we get the following result (cf. Corollary \ref{7.3}) which is obtained 
as a consequence of Kato's Localization Theorem \ref{7.2}. 
 \begin{corollary} \label{1.9}For $x,y \in W$ and $\beta_1, \beta_2 \in Q_{<0}^\vee $ , in the quantum product
$$
\left[ \mathcal{O}^x\right] * \left[ \mathcal{O}^y\right] = \sum_{\beta\leq 0, \, z\in W'_\beta} p_{x\tau_{\beta_1},y \tau_{\beta_2}}^{z\tau_\beta}  q^{\beta- \left( \beta_1+\beta_2\right)} \left[\mathcal{O}^z\right] \in QK_T(X),
$$
where $Q_{<0}^\vee := \left\{ q\in Q^\vee:~ \alpha_i (q) <0, \mbox{ for all the simple roots $\alpha_i$ of $G$}\right\},$
$p_{x\tau_{\beta_1},y \tau_{\beta_2}}^{z\tau_\beta}$ are the structure constants as above for the modified convolution product $\odot$ in $K_0^T \left(\X\right)$, $W_\beta $ is the stabilizer of $\beta $ in $W$ and $W_\beta'$ is the set of minimal coset representatives in $W/W_\beta$.  
\end{corollary}
For $x,y\in W$, write the quantum product in $QK_T(X)$: 
$$
 \left[ \mathcal{O}^x\right] *\left[ \mathcal{O}^y\right] = \sum_{z\in W, \, \eta \in Q_+^\vee} d_{x,y}^{z,\eta} q^\eta \left[\mathcal{O}^{z}\right].
$$
The above  Conjecture \ref{1.6} is equivalent to the following conjecture on the quantum product structure constants in $QK_T(X)$ (cf. Proposition \ref{7.7}). 
\begin{conjecture}\label{1.10}
For any $x,y, z \in W$ and $\eta \in Q_+^\vee$, 
$$
(-1)^{\ell(x)+\ell(y)-\ell(z)} d^{z,\eta}_{x,y} \in \mathbbm{Z}_+\left[ \left(e^{\alpha_1}-1\right),\ldots , \left(e^{\alpha_l}-1\right)\right].
$$
\end{conjecture}

We mention some of the known positivity results or conjectures related to $QK(X)$ and $QK_T(X)$ by  Lenart-Maeno [LM], Buch-Mihalcea [BM-1], Lam-Schilling-Shimozono [LSS], Li-Mihalcea [LiM], Buch-Chaput-Mihalcea-Perrin [BCMP-1] and [BCMP-2], 
Lenart-Naito-Sagaki [LNS], Xu [Xu] and  Benedetti-Perrin-Xu [BPX]. For more details, see Remark \ref{7.9}.

\vskip1ex

\noindent{\bf Acknowledgements:} I am very grateful to Syu Kato for numerous correspondences and conversations who patiently explained to me his works, especially [Ka-1]. Part of this work was done while the author was visiting the Institute for Advanced Study, Princeton during the fall semester of 2022 and the Institut des Hautes \'Etudes Scientifiques, Bures-sur-Yvette during the fall semester of 2023. I gratefully acknowledge their support. I also thank L. Mihalcea for some of the references.

\section{ Formulation of the main conjecture}

\noindent 

Let $G$ be a connected simply-connected simple algebraic group over $\mathbb{C}$. We fix a Borel subgroup $B$ and a maximal torus ${T}\subset {B}$. We also fix a maximal compact subgroup $K$ of $G$ such that $T_o:=T\cap K$ is a (compact) maximal torus of $K$. 
Let $\mathbb{C}((t))$ be the field of Laurent power series and let $\G:=G((t))$ be the loop group consisting of $\mathbb{C}((t))$ rational point of $G$. Let $\P:=G[[ t]]$ be the standard maximal parahoric subgroup, which is the set of $\mathbb{C}[[t]]$ rational points of $G$. 
Consider the {\it affine Grassmannian}  $\X=\G/\P$. Then, $\X$ is an ind-projective variety with filtration $\X_0\subset \X_1\subset \ldots \subset\X_n \subset\ldots $ given by Schubert varieties: 
$$
\X_n=\bigcup_{\{w\in\W': \ell(w)\leq n\}} \,\,\B w\P/ \P,
$$
where $\B$ is the standard Iwahori subgroup defined as  the inverse image of $B$ in $G[[t]]$ under the evaluation map at $t=0$, $\W:=W\propto Q^\vee$ is the affine Weyl group, $Q^\vee$ is the coroot lattice of $G$, $W$ is the (finite) Weyl group of $G$ and $\W'$ is the set of minimal length coset representatives in $\W/W$. In particular, $\X$ has inductive limit analytic topology. The torus $T$ acts on $\X$ via the left multiplication keeping each $\X_n$ stable. Define 
$$
K_T^{\top}\left(\X\right)= \underset{n\rightarrow \infty}{\Inv.lt.}\,\,
K_T^{\top}\left(\X_n\right).
$$

Observe that $\X\simeq G\left[t^{\pm 1}\right]/G[t]$, where we abbreviate $G\left(\mathbb{C}[t^{\pm}]\right)$ by $G\left[t^{\pm1}\right]$ etc.

Let $\bar{\X}:= G((t^{-1}))/G[t]$ be the 
{ \it thick} loop group, where $\mathbb{C} ((t^{-1})):=\mathbb{C} [[t^{-1}]][t]$ is viewed as the set of Laurent series in $t^{-1}:\left\{ \sum\limits_{n\leq k} a_n t^n, ~ a_n\in \mathbb{C}\right\}$. 

\begin{definition} \label{2.1}\rm{
For a quasi-compact scheme $\Y$, an $\mathcal{O}_{\Y}$-module $\uS$ is called {\it coherent} if it is finitely presented as an $\mathcal{O}_\Y$-module and any $\mathcal{O}_\Y$-submodule of finite type admits a finite presentation. 

A subset $S\subset\W'$ is called an ideal if $x\in S$ and $y\leq x$ in $\W'$ imply $y\in S$. An $\mathcal{O}_{\bar{\X}}$-module $\T$ is called coherent if $\T_{|{\V}^S}$ is a coherent $\V^S$-module for any finite ideal $S\subset \W'$, where $\V^S$ is the quasi-compact open subset of $\bar{\X}$ defined by $\V^S:= \bigcup_{w \in S}\,w \bar\U^- \underline{o}$, where $\underline{o}$ is the base point of $\bar{\X}$ and $\bar\U^-:= G[[t^{-1}]]$. Then,
$\V^S = \bigcup _{w \in S}\,\bar\U^- w \underline{o}$; in particular, $\V^S$ is $\bar\U^-$-stable. }
\end{definition}
We recall the following result due to Kashiwara-Shimozono [\cite{KS}, Lemma 8.1].

\begin{lemma}\label{2.2}
For any $T$-equivariant coherent sheaf $\uS$ over $\bar{\X}$ and any finite ideal $S\subset \W$, the sheaf $\uS_{| \V^S}$ admits a finite resolution by locally free sheaves $\F_i$ over $\V^S$: 
$$
0\rightarrow \F_k \rightarrow \cdots \rightarrow \F_2 \rightarrow \F_1 \rightarrow \F_0 \rightarrow \uS_{|\V^S} \rightarrow 0. 
$$
Moreover, for any $ w\in \W'$, the sheaf $\xi^w:=\mathcal{O}_{ X^w} \left( -\partial  X^w \right)$ over $\bar{\X}$ is a coherent sheaf, where 
\begin{equation*}
 C^{w} := \bar\U^- w \underline{o},\,\,\,
 X^w := \overline{C^w} \subset \bar{\X} ~\mbox{and}\,\,
\partial  X^w :=  X^w \backslash  C^w.\quad \square 
\end{equation*} 
\end{lemma}

Let $K^0_T \left( \bar{\X}\right)$ denote the Grothendieck group of $T$-equivariant coherent $\mathcal{O}_{\bar{\X}}$-modules. Thus, $K^0_T\left(\bar{\X}\right)$ can be thought of as the inverse limit of $K^0_T\left(\V^S\right)$, as $S$ varies over the finite ideals of $\W'$. 
For any $w \in \W'$, the $K$-theory class of the coherent $\mathcal{O}_{\bar{\X}}$-module $\xi^w$ is denoted by 
$$
\left[\xi^w\right] \in K^0_T \left( \bar{\X}\right). 
$$
In particular, we can also think of $\left[\xi^w\right]$ as an element $\left[\bar{\xi}^w\right]$ of $K^{\top}_T\left( {\X}\right)$ by using Lemma \ref{2.2}. 

We also define the {\it $T$-equivariant $K$-homology} $K^T_0 \left( {\X}\right)$ by 
$$
K^T_0 \left( {\X}\right) = \underset{n\rightarrow \infty}{\dir.lt.}\,\, K^T_0 \left( {\X}_n\right), 
$$
where $K^T_0 \left( {\X}_n\right)$ is the Grothendieck group corresponding to the $T$-equivariant coherent sheaves on the projective variety $\X_n$.

\begin{definition}\label{2.3}\rm{
Consider the $R(T)$-bilinear pairing 
$$\langle ~,~ \rangle: K^T_0 \left( \bar{\X}\right) \otimes_{R(T)} K^T_0 \left( {\X}\right) \rightarrow R(T)$$
 defined by 
$$
\left\langle \left[\mathcal{S}\right],~\left[\mathcal{F}\right] \right\rangle = \sum_{i} (-1)^i  \chi_T \left( \X_n,~ \Tor_i^{\mathcal{O}_{\bar{\X}}} \left(\mathcal{S},\mathcal{F}\right)\right), 
$$ 
for $\mathcal{S}$ a $T$-equivariant coherent sheaf on $\bar{\X}$ and $\mathcal{F}$  a $T$-equivariant coherent sheaf on $\X$ supported in $\X_n$ (for some $n$), where $\chi_T$ denotes the $T$-equivariant Euler-Poincar\'e characteristic and $R(T)$ is the representation ring of $T$ over $\mathbb{Z}$.  }
\end{definition}

We recall the following theorem due to Compton-Kumar [\cite{CK}, Proposition 3.8].

 \begin{theorem}\label{2.4} 
 Under the above pairing, for any $v,  w\in \W'$, 
$$
\left\langle \left[ \xi^v\right], ~\left[ \mathcal{O}_{X_w}\right]\right\rangle = \delta _{v,w},
$$
where the finite dimensional Schubert variety 
$$
X_w := \overline{\B w \underline{o}} \subset \X.
$$
$\square$
\end{theorem}

Let $$\Omega (K) :=\left\{\gamma: S^1 \rightarrow K: \gamma (1) =1\,\text{and $\gamma$  extends to an algebraic morphism $\tilde{\gamma}: \mathbb{C}^* \rightarrow G$}\right\}$$
be the based {\it algebraic}  loop group of $K$. 
Then, $K$ (in particular, $T_0$) acts on $\Omega (K)$ via conjugation: 
$$\left(k\cdot \gamma\right) (z)= k\gamma (z)k^{-1},\,\,\text{ for $k\in K,  \gamma\in \Omega (K)$ and $z\in S^1$}.$$

Choose an embedding $\rho$:
$$
K\subset G \overset{\rho} {\hookrightarrow}
GL_N(\mathbbm{C}) \subset M_N(\mathbbm{C}).
$$
We endow $\Omega (K)$ with the 
   inductive limit topology induced from the filtration: 
$$
\Omega (K)_1 \subset \Omega(K)_2 \subset \cdots \subset \Omega (K)_n \subset \cdots ,
$$
where
\begin{eqnarray*}
\Omega (K)_n &:=& \left\{\vphantom{\sum_{k=-n}^n a_k^{i,j}} \gamma:S^1 \rightarrow K:\, \rho(\gamma)\mbox{ has its $(i,j)$-th matrix entry of the form } \right.\\
&&\left.\sum_{k=-n}^n a_k^{i,j}
 z^k\mbox{ with } a_k^{i,j}  \in \mathbb{C} \mbox{ for } z \in S^1\right\}
\end{eqnarray*}
is realized as a closed subset of $\mathbbm{C}^{(2n+1)N^2}$ (under the analytic topology) coming from the coefficients $\left\{ a_k^{i,j}\right\}$. Then, this topology on $\Omega (K) $ does not depend upon the choice of the embedding $\rho$.

The following lemma is well-known (cf. [\cite{PS}, \S3.5 and Theorem 8.6.3]). 

\begin{lemma} \label{2.5}
The inclusion map 
$$ 
\beta: \Omega (K) \rightarrow \X,~ \gamma \mapsto \tilde{\gamma}\cdot \underline{o},\quad \mbox{for } \gamma\in \Omega (K)
$$
is a $K$-equivariant homeomorphism  under the above topology on $\Omega (K)$ and the analytic topology on $\X$.  

\end{lemma}

\begin{definition} \label{2.6}\rm{
Consider the multiplication map 
$$
\hat{m}:~\Omega (K) \times \Omega (K) \rightarrow \Omega (K), \quad \left( \gamma_1,  \gamma_2\right) \mapsto \gamma_1\cdot \gamma_2.
$$

From the above description of the topology on $\Omega(K)$, it is easy to see that $\hat{m}$ is continuous. Moreover, $\hat{m}$ is $K$-equivariant (in particular, $T_0$-equivariant) under the conjugation action of $K$ on $\Omega(K)$ viewing the elements of $K$ as constant loops and acting diagonally on the domain of $\hat{m}$. 

By virtue of the $K$-equivariant homeomorphism $\beta$ (cf. Lemma \ref{2.5}), we get a $K$-equivariant {\it continuous} map 
$$
m:~\X\times \X \rightarrow \X.
$$

Thus, we get a pull-back map
$$
m^*:~K^{\top}_{T_0}\left(\X\right) = K^{\top}_T \left(\X\right) \rightarrow K^{\top}_T \left(\X\times \X\right) = K_T^{\top} \left(\X\right) {\hat{\otimes}}_{R(T)} K_T^{\top} \left(\X\right),
$$
where 
$$
K_T^{\top}\left(\X\right)  {\hat{\otimes}}_{R(T)}\, K_T^{\top} \left(\X\right) := \underset{n\rightarrow \infty}{\Inv.lt.} \,\,\left(K_T^{\top}\left(\X_n\right)\otimes_{R(T)}K_T^{\top} \left(\X_n\right)\right). 
$$

Observe that since $T/T_0$ is contractible, $K_{T_0}^{\top}\left(\X\right)=K_T^{\top}\left(\X\right)$. Moreover,  since $K_T^{\top}\left(\X_n\right)$ is a free $R(T)$-module (cf. [\cite{KK}, Proof of Lemma 3.15]), by the Kunneth theorem [\cite{Mc}, Theorem 4.1], 
$$
K_T^{\top}\left(\X_n\times \X_n\right) \approx K_T^{\top}\left(\X_n\right) \otimes_{R(T)} K_T^{\top}\left(\X_n\right).
$$ 

Recall that (cf. [\cite{CK}, Proposition 3.5])
\begin{equation}\label{eq2.6.1}
K_T^{0}\left(\bar{\X}\right)= \Pi_{w\in \W'} \,\,R(T)\left[\xi^w\right].
\end{equation}
By virtue of the above result,  we call $\left\{ \left[\xi^w\right]\right\}_w$ an {\it infinite} basis.

 Also, by [\cite{CK}, Lemma 3.2],
\begin{equation}\label{eq2.6.2}
K_0^T\left(\X\right) = \bigoplus_{ w\in \W'} R(T)\cdot \left[\mathcal{O}_{X_w}\right].
\end{equation}}
\end{definition}

\begin{lemma}\label{2.7}
The canonical map 
$$
i_{\X}:~ K^0_T\left(\bar\X\right) \rightarrow K_T^{\top} \left(\X\right)
$$
is an $R(T)$-algebra isomorphism. Thus, 
$$
K_T^{\top} \left(\X\right) = \Pi_{w\in \W'} \,R(T) \left[\bar{\xi}^{w}\right].
$$
\end{lemma}
\begin{proof}
Let $\Y:= \G/\B$. Then, by [\cite{KK}, Proposition 3.39] together with [\cite{Ku-2}, Proposition 3.6] (since the Schubert varieties in $\Y$ have rational singularity [\cite{Ku-1}, Theorem 8.2.2 (c)]), we obtain that the canonical map 
$$
i_\Y: K^0_T \left(\bar{\Y}\right) \rightarrow K_T^{\top}\left(\Y\right)\mbox{ is an isomorphism}, 
$$
where $\bar{\Y} := G((t^{-1}))/G[t]\cap \B$.

Let $\pi: \bar{\Y} \rightarrow \bar{\X}$ be the standard projection. 
Then, 
$$R^i\pi_*\left(\mathcal{O}_{\bar{\Y}}\right)=0, \,\,\,\text{for all $i>0,$}$$
 since $H^i\left(G/B,\mathcal{O}_{G/B}\right)=0$, for all $i>0$. Thus, using the projection formula [\cite{Ha}, Chap. III, Exercise 8.3] we get that the induced map
$$\pi^*: K^0_T\left(\bar{\X}\right)\rightarrow K^0_T \left(\bar{\Y}\right)$$ is injective.

Since $i_\Y$ is injective, we get that so  is $i_\X$ from the following commutative diagram:
\begin{equation*}
\begin{array}{lll}
\displaystyle K_T^0\left(\bar{\X}\right) &\displaystyle \xrightarrow{i_{\X}}&\displaystyle  K_T^{\top}\left(\X\right)\\
\,\,\,\,\,\,\displaystyle \lhookdownarrow \pi^* &&\,\,\,\,\,\,\displaystyle \downarrow \pi^*\\
\displaystyle K_T^0\left(\bar{\Y}\right)&\displaystyle \underset{i_{\Y}}{\xrightarrow{\sim}} & \displaystyle K_T^{\top} \left(\Y\right).
\end{array}
\end{equation*}
We next prove that $i_\X$ is surjective:

By [\cite{KK}, Corollary 3.20 and Lemma 2.27], we get that $\left\{ \L\left(\hat{\rho}\right)\cdot \left[\xi^w_\B\right]\right\}_{w\in \W'}$ is an infinite $R(T)$-basis of $K_T^{\top} \left(\X\right)$, where $\L\left(\hat{\rho}\right)$ is the line bundle over $\Y\approx \hat{\G}/\hat{\B}$ corresponding to the character $e^{-\hat{\rho}}$ of $\hat\B$. Here $\hat\G$ is the universal central extension of $\G$ (cf. [\cite{Ku-3}, Definition 1.4.5 corresponding to $\lambda_c=0_1$]), $\hat\B$ is the inverse image of $\B$ in $\hat\G$,  $\hat\rho$ is the weight taking value $1$ on each of the affine simple coroots $\left\{\alpha ^\vee_i\right\}_{0\leq i \leq l}$ and $\xi^w_\B:= \mathcal{O}_{X_\B^w}\left(-\partial X_\B^w\right)$, $X_\B^w = \overline{\bar\U^- w {\underline{o}}_\B}, \partial  X^w_\B:=  X^w_\B \setminus \bar\U^- w {\underline{o}}_{ \B}$, $\underline{o}_\B$ being the base point of $\Y$. (We have used here [\cite{KK}, Proposition 3.9] and [\cite{CK}, Proposition 3.8] to transform the basis in [\cite{KK}] to our basis $\left[\xi^w\right]$.) 

By [\cite{KK}, Proposition 2.22] considering the localization map 
$$
r: K_T^{\top}\left(\X\right) \rightarrow K_T^{\top} \left(\X^T\right),
$$
we get for any $v\in \W'$, 
$$
\left[\xi^v\right] = \sum_{ w\in \W'} a^v_w \left( \L (\hat{\rho})\cdot\left [\xi_\B^w\right]\right) \in K_T^{\top}\left(\X\right), 
$$
where
$$a^v_w=
\left\{ \begin{array}{l}
0,\,\,\,\mbox {for} \,\,\ell(w)\leq \ell(v)\mbox{ and }w\neq v 
\\
e^{v^{-1}\hat{\rho}}, \,\,\,\mbox {for} \,\, w = v.
\end{array}\right.
$$
Thus, the matrix $A=\left(a^v_w\right)_{v, w\in \W'}$ with entries in $R(T)$ is an upper triangular matrix with invertible diagonal entries. In particular, $A$ is an invertible matrix. 
This shows that, for any $v\in\W'$, 
$$
\L\left(\hat{\rho}\right)\cdot \left[\xi^v_B\right] \in \Pi_{w\in \W'}\,\,R(T) \,\,\left[\xi^w\right]. 
$$

Hence,
$$K_T^{\top}\left(\X\right)=\Pi_{w\in\W'}\,\,R(T)\left[\xi^w\right].$$
 This proves the surjectivity of $i_{\X}$ in view of \eqref{eq2.6.1} and hence the lemma is proved. 
\end{proof}

\begin{remark}\label{2.8}\rm{
The map $m:\X\times \X \rightarrow \X$ is {\it not} an algebraic morphism with respect to the ind-variety structure on $\X$. In fact, it fails to be an algebraic morphism already for $G=\SL_2(\mathbb{C})$. }
\end{remark}

The induced map 
$$
m^*:~ K_T^{\top}\left(\X\right) \rightarrow K_T^{\top} \left(\X\right)  {\hat{\otimes}}_{R(T)}\, K_T^{\top} \left(\X\right)
$$
as in Definition \ref{2.6} can be written as follows by using Lemma \ref{2.7} (for any $w\in\W'$):
$$
m^*\left(\left[\bar{\xi}^w\right]\right) = \sum_{u,v\in \W'} a^w _{u,v}\left[\bar{\xi}^u\right] \otimes\left [\bar{\xi}^v\right],\mbox{ for unique } a^w_{u,v}\in R(T).
$$

The following is our main conjecture. 

\begin{conjecture}\label{2.9}
We conjecture that for any $u,v,w\in\W'$, 
$$
(-1)^{\ell(u)+\ell(v)-\ell(w)} a^w_{u,v}\in \mathbb{Z}_+\left[\left(e^{\alpha_1}-1\right), \ldots , \left(e^{\alpha_l}-1\right)\right],
$$
 i.e., $(-1)^{\ell(u)+\ell(v)-\ell(w)} a^w_{u,v}$ is a polynomial in the variables $x_1=e^{\alpha_1}-1,\ldots,x_l=e^{\alpha_l}-1$ with non-negative integral coefficients, where $\left\{\alpha_1,\ldots,\alpha_l\right\}$ are the simple roots of $G$. 
   $\square$
\end{conjecture}

\begin{remark}\label{2.10}\rm{
Considering the localization of $\bar{\xi}^w$, it is easy to see that, in the above sum, }
$$ 
a^w_{u,v} =0 \quad \mbox{unless} \quad \ell(u)+\ell(v) \geq \ell(w).
$$ 
\end{remark}
\section{An equivalent formulation of the main conjecture in terms of Pontryagin product}

\begin{definition}[Pontryagin Product]\label{3.1}\rm{
 Recall from Definition \ref{2.6} the multiplication map 
$m: \X\times \X \rightarrow \X$ via the identification of $\X$ with $\Omega(K)$ under $\beta$. This gives rise to the pull-back map
$$
m^*:~ K_T^{\top}\left(\X\right) \rightarrow K_T^{\top} \left(\X\times \X\right).
$$

By Lemma \ref{2.7}, we have a canonical isomorphism 
$$
i_{\X}: ~K^0_T\left(\bar{\X}\right) \xrightarrow{\sim} K_T^{\top}\left(\X\right),
$$
and a similar isomorphism (by the same proof)
$$
i_{\X\times \X} : K^0_T\left(\bar{\X}\times\bar{\X}\right)\xrightarrow{\sim} K_T^{\top}\left(\X\times \X\right).
$$

Thus, the map $m^*$ gives rise to the map
$$
\tilde{m}^*:~ K^0_T\left(\bar{\X}\right) \rightarrow K^0_T\left(\bar{\X}\times\bar{\X}\right)
$$
under the identifications $i_{\X}$ and $i_{\X\times \X}$. 

Now, the pairing (over $R(T)$)
$$
\langle ~,~\rangle: ~ K^0_T\left(\bar{\X}\right) \otimes_{R(T)}  K_0^T\left({\X}\right) \rightarrow R(T)
$$
as in Definition \ref{2.3} is non-singular by Theorem \ref{2.4}. This induces the identification: 
$$
\psi :  K^0_T\left(\bar{\X}\right) \simeq \Hom_{R(T)} \left(K_0^T\left({\X}\right) ,R(T)\right)
$$
and a similar identification
$$
\tilde{\psi}:~K^0_T\left(\bar{\X}\times\bar{\X}\right)\simeq \Hom_{R(T)} \left(K_0^{T}\left(\X\times \X\right) ,R(T)\right).
$$
Using these identifications $\psi$ and $\tilde{\psi}$, we can rewrite the map $\tilde{m}^*$ as 
$$
\hat{m}^*:\Hom_{R(T)} \left(K_0^T\left({\X}\right) ,R(T)\right) \rightarrow \Hom_{R(T)} \left(K_0^{T}\left(\X\times \X\right) ,R(T)\right)
$$
giving rise to the product 
$$
{\p}:~K_0^{T}\left(\X\times \X\right) \simeq K^T_0\left({\X}\right) \otimes_{R(T)} K_0^T\left({\X}\right) \rightarrow K_0^T\left({\X}\right),
$$
where the first identification follows from the identity \eqref{eq2.6.2} for $\X$ and a similar identity for $\X$ $\times$ $\X$. Moreover, the image of the map $\p$ lands inside 
$K_0^T\left({\X}\right) \subset \left(K^0_T\left({\X}\right)^*\right)^*$ due to Remarks \ref{2.10}, where, for an $R(T)$-module $M$, 
$$
M^* := \Hom_{R(T)} (M,R(T)).
$$
Thus, $\p$ makes $K_0^T\left({\X}\right)$ into an $R(T)$-algebra. Its product is called the {\it Pontryagin product}. Let us write,  under the Pontryagin product, 
for $u,v \in \W'$,
\begin{equation}\label{3.1.1}
\left[\mathcal{O}_{X_u}\right]* \left[\mathcal{O}_{X_v}\right]= \sum_{ w\in \W'}b_{u,v}^w\left [\mathcal{O}_{X_w}\right]. 
\end{equation}

Then, by Remark \ref{2.10},}
 $$
 b_{u,v}^w =0 \quad \mbox{if } \ell(w) > \ell(u) +\ell(v).
 $$
 \end{definition}
 
Moreover, by Theorem \ref{2.4}, we get the following. Also, see [LSS, $\S$5.1], where they define their $K_T(\X)$ as the continuous dual of $K_T^{\top}(\X)$, which is equivalent to our definition of $K^T_0(\X)$ in view of  Theorem \ref{2.4}. 
 \begin{lemma} \label{3.2}
For any $u,v,w\in\W'$, 
$$
a_{u,v}^w = b_{u,v}^w. \qquad \qquad\qquad
\square
$$
  \end{lemma}

Thus, Conjecture \ref{2.9} translates to the following equivalent conjecture on the Pontryagin product in $K_0^T\left({\X}\right)$. 

\begin{conjecture}\label{3.3} Under the Pontryagin product as above, its structure constants $b_{u,v}^w$ satisfy 
$$
(-1)^{ \ell(u)+\ell(v)-\ell(w)}
\,\,b_{u,v}^w \in \mathbbm{Z}_+ \left[\left(e^{\alpha_1}-1\right), \ldots , \left(e^{\alpha_l}-1\right)\right] ,
$$
where $\mathbbm{Z}_+ \left[\left(e^{\alpha_1}-1\right), \ldots , \left(e^{\alpha_l}-1\right)\right]$ denotes polynomials in $\left(e^{\alpha_1}-1\right), \ldots , \left(e^{\alpha_l}-1\right)$ with non-negative integral coefficients.  $\square$
\end{conjecture}

\section{Pontryagin product in terms of convolution product}

Consider the diagram
$$
\begin{array}{lll}
\tilde{\X}:=& \G \times^{\B} \X & \xrightarrow{\mu} \X,\\
&\quad \downarrow \pi&\\
& \,\,\, \Y
\end{array}
$$
where $\Y:=\G/\B$, $\mu([g,x]):=g\cdot x$ and $\pi ([g,x]):= g\B$ for $g\in \G$ and $x\in \X$. 

Observe that both $\mu$ and $\pi$ are $\G$-equivariant morphisms under the left action of $\G$ on the spaces involved.
\begin{definition}\label{4.1} \rm{
Take $\B$-equivariant coherent sheaves $\mathcal{S}_1$ on $\Y$ and $\mathcal{S}_2$ on $\X$ supported in $p^{-1}(\X_n)$ and $\X_n$  respectively (for some $n>0$), where $p:\Y \rightarrow \X$ is the projection. Their {\it convolution product} is defined by 
$$
\mathcal{S}_1 \odot' \mathcal{S}_2 := \mu_! \left( \left( \pi^* \mathcal{S}_1\right) \otimes^L \left( \epsilon \boxtimes^{\B}\mathcal{S}_2\right)\right) \in K_0^\B \left(\X\right),
$$
where $\epsilon \boxtimes^{\B}\mathcal{S}_2$ denotes the sheaf on $\G\times^{\B} \X$ the pull-back of which to $\G \times \X$ is the product sheaf $\epsilon \boxtimes\mathcal{S}_2$ 
($\epsilon$ being the rank-1 trivial bundle over $\G$) (cf. [\cite{SGA1},  Chap. VIII, \S1]),
 $\otimes^L$ is the derived tensor product $\sum(-1)^i \Tor_i^{\mathcal{O}_{\tilde{\X}}}$ and $\mu_! :=\sum_i (-1)^i R^i \mu_*$.

Observe that $\left(\pi^* \mathcal{S}_1\right) \otimes^L \left(\epsilon \boxtimes^{\B}\mathcal{S}_2\right) $ is well defined; in fact, 
 \begin{equation} \label{4.1.1}\Tor_i^{\mathcal{O}_{\tilde{\X}}} 
 \left(\pi^* \mathcal{S}_1,\epsilon \boxtimes^{\B}\mathcal{S}_2\right)=0, \,\,\,\text{for all $i>0$},
 \end{equation}
  as can be easily seen by pulling the two sheaves to $\G$ $\times$ $\X$. 
Further, the sheaf $\left(\pi^* \mathcal{S}_1\right) \otimes_{\mathcal{O}_{\tilde{\X}}}
 \left(\epsilon \boxtimes^{\B}\mathcal{S}_2\right) $ has support in a projective variety (of finite dimension), and hence $\mu_!$ is well defined. 

Since $\mu_!$ and $\otimes^L$ both descend to corresponding $K$-groups, we get a well defined map
$$
\odot ': ~ K_0^\B(\Y) \otimes _\mathbbm{Z} K_0^\B \left(\X\right) \rightarrow \K_0^\B \left(\X\right). 
$$

Observe that $\odot '$ is $R(\B)$-linear in the first variable but, in general, {\it not} $R(\B)$-linear in the second variable but it is $R(\P)$-linear (cf. Corollary \ref{4.5}). }
\end{definition}

For generalities on convolution product, we refer to [\cite{CG}, \S5.2]. 
\begin{definition}\label{4.2}
\rm{
In any Coxeter group $\W$,  define the {\it Demazure product} $*$ for any $u\in \W$ and simple reflection $s_i$,
$$
u* s_i=
\left\{ \begin{array}{l}
u,\,\,\,  \mbox{if}\quad us_i <u\\
us_i,\,\,\, \mbox{if}\quad us_i >u.
\end{array}\right.
$$
This extends to an associative product by defining 
$$u*v= (\cdots (( u*s_{i_1})*s_{i_2})\cdots *s_{i_n})$$
 for a reduced decomposition $v=s_{i_1}\ldots s_{i_n}$. (It does not depend upon the choice of the reduced decomposition of $v$.)}
\end{definition}

\begin{proposition} \label{4.3}
For $u\in \W$ and $v\in \W'$,
$$
\left[\mathcal{O}_{X^\B_u}\right] \odot'\left [\mathcal{O}_{X_v}\right] =\left [\mathcal{O}_{X_{\overline{u*v}}}\right] \in K_0^B\left(\X\right) ,
$$
where ${X^\B_u}:= \overline{\B u\B/\B} \subset \Y$.

Observe that $u*v$ may not lie in $\W'$. We take its unique representative $\overline{u*v}$ in $\W'$. 
\end{proposition}

\begin{proof}
As observed in identity \eqref{4.1.1}, following its notation, 
$$
\Tor^{\mathcal{O}_{\tilde\X}}_i \left(\pi^* \mathcal{O}_{X_u^\B}, \epsilon\boxtimes^{\B}\mathcal{O}_{X_v}\right)=0,  \quad \mbox{for}\quad i>0.
$$

Further, 
\begin{eqnarray}\label{eq4.3.1}
\left(\pi^* \mathcal{O}_{X_u^\B}\right) \otimes _{\mathcal{O}_{\tilde\X}}\left(\epsilon\boxtimes^{\B}\mathcal{O}_{X_v}\right) &=&  \left(\mathcal{O}_{\tilde{p}^{-1}\left(X_u^\B\right)}\boxtimes^{\B}\mathcal{O}_{\X}\right) \otimes
\left(\mathcal{O}_\G\boxtimes^{\B}\mathcal{O}_{X_v}\right)\nonumber\\
&=&\mathcal{O}_{\tilde{p}^{-1}\left(X_u^\B\right)}\boxtimes^{\B}\mathcal{O}_{X_v},
\end{eqnarray}
where $\tilde{p}:\G\rightarrow \Y$ is the standard projection. 

Thus,
\begin{equation}\label{eq4.3.2}
\left[\mathcal{O}_{X_u^\B}\right] \odot' \left[\mathcal{O}_{X_v}\right] = \mu_!\left(\mathcal{O}_{\tilde{p}^{-1}\left(X_u^\B\right)}\times^{\B}\mathcal{O}_{X_v}\right)=
\mu_! \left(\mathcal{O}_{(\tilde{p}^{-1}\left(X_u^\B\right)\boxtimes^{\B} X_v)}\right) .
\end{equation}

Take a reduced decomposition
$u=s_{i_1}\ldots s_{i_n}$, where $\{s_i\}_{0\leq i\leq l}$ are the simple reflections of $\W$. 
Let
$$\mathcal{Z}_u' := \P_{i_1}\times^\B \P_{i_2}\times ^\B \ldots \times ^\B \P_{i_n}$$
 be the BSDH (Bott-Samelson-Demazure-Hansen) variety, where $\P_i\supset \B$ is the minimal parabolic subgroup of $\G$ containing $s_i$ (cf. [\cite{Ku-1}, \S 7.1.3]). Then, we have a morphism
$$
\beta_u':\mathcal{Z}'_u\rightarrow \tilde{p}^{-1}\left(X_u^\B\right),\,\,\,[p_1,\ldots ,p_n] \mapsto p_1p_2 \ldots p_n, \mbox { for } p_j\in \P_{i_j}. 
 $$
 Similarly, let $\beta_v:\mathcal{Z}_{v}\rightarrow X_v$ be a BSDH desingularization (cf. [Ku-1, $\S$7.1.3]). Then, we have the commutative diagram induced from the morphisms $\beta'_u$ and $\beta_v$: 
$$
\begin{array}{ll}
\mathcal{Z}_{u,v}= \mathcal{Z}'_u & \times^\B \mathcal{Z}_v \xrightarrow{\beta_u' \times \beta_v} \tilde{p}^{-1}\left(X^\B_u\right) \times^\B X_v\\
&\rotatebox[origin=c]{140}{$\xleftarrow{{\beta_{u,v}}}$} 
\phantom{X^\B  \underrightarrow{\beta_u' \times \beta_v} }\downarrow^\mu\\
&\phantom{= X^\B \mathcal{Z}_v \underrightarrow{\beta_u' \times \beta_v} } X_{\overline{u*v}} .
\end{array}
$$
Observe that, for any sequence of simple reflections  $\underline{s} = (s_{j_1},\ldots,s_{j_m})$ in $\W$, 
\begin{equation}\label{eq4.3.3}
\Image \left(\beta_{\underline{s}}\right)= X_{\overline{s_{j_1}*s_{j_2}* \ldots * s_{j_m}}}\mbox{ (cf. [\cite{Ku-1}, Theorem 5.1.3 and Definition 7.1.13])}.
\end{equation}
By [\cite{Ku-1}, Theorem 8.1.13] for $M=\mathbbm{C}$,
\begin{equation}\label{eq4.3.4}
R^i \beta_* \left({\mathcal{O}_{\mathcal{Z}}}_{u,v}\right)=0, \quad\mbox{for } i>0 
\end{equation}
and 
\begin{equation}\label{eq4.3.5}
\beta_* \left({\mathcal{O}_{\mathcal{Z}}}_{u,v}\right)= \mathcal{O}_{X_{\overline{u*v}}},\quad \mbox{where } \beta:= \beta_{u,v},
\end{equation}
since $X_{\overline{u*v}}$ is normal by [\cite{Ku-1}, Theorem 8.3.2(b)]. A similar property as \eqref{eq4.3.4}  and \eqref{eq4.3.5} is true for the morphism $\beta_u'\times \beta_v$. Thus, by the Grothendieck spectral sequence for the composition of two functors (cf. [\cite{Ja}, Part I, Proposition 4.1]), we get 
$$
\left(R^i \mu_*\right)\left(\mathcal{O}_{(\tilde{p}^{-1}(X_u^\B)\times^\B X_v})\right) =
\left\{ \begin{array}{l}
0, \qquad for ~i>0\\
\mathcal{O}_{X_{\overline{u*v}}}, \qquad for ~i=0.
\end{array}\right.
$$
This proves the proposition by using \eqref{eq4.3.2}. $\square$
\end{proof}

As before, let $\B\subset \P_i ~(0\leq i \leq l)$ denote the minimal parabolic subgroup of $\G$ containing the simple reflection $s_i$. 
\begin{proposition} \label{prop4.4}
Let $\mu_i:\P_i\times^\B * \rightarrow\P_i/\B$ be the map $[p,*]\mapsto p\B$,  for $p\in \P_i$. Then, for any character $e^\lambda$ of $\B$,
$$
(\mu_i)_! \left(\mathcal{O}_{X_i}\boxtimes^\B e^\lambda\right)= e^{s_i\lambda}\left[\mathcal{O}_{X_i}\right]+ \left(\frac{e^\lambda- e^{s_i\lambda}}{1-e^{\alpha_i}}\right)\left[\mathcal{O}_{e}\right]\in K_0^\B(X_i),
$$
where $X_i:=\P_i/\B \simeq \mathbb{P}^1$.

Here $s_0$ is thought of as $s_\theta$ (reflection corresponding to the highest root $\theta$ of $G$) and $\alpha_0:=-\theta$. Observe that $\mu_{i!} \left(\mathcal{O}_{X_i}\boxtimes^\B e^\lambda\right)= \L_{X_i}(-\lambda)$. 
\end{proposition}

\begin{proof}
Write in $K_0^\B(X_i)$ 
$$
\left[\L_{X_i}(\lambda)\right] = a_\lambda\left[\mathcal{O}_{X_i} \right] + b_\lambda \left[\mathcal{O}_e\right]
,\mbox{for } a_\lambda,b_\lambda \in K_0^\B (*).
$$
Take a character $e^\mu$ of $\B$ such that $m:= \mu(\alpha_i^\vee)>0$ and $n+m\geq 0$, where $n:=\lambda \left(\alpha_i^\vee\right)$ and $\alpha_0^\vee:=-\theta^\vee$. Then, 
\begin{equation}\label{4.4.1}
 \left[\L_{X_i}(\lambda+\mu)\right] = a_\lambda\left[\L_{X_i}(\mu) \right] + b_\lambda e^{-\mu} \in K_0^\B(X_i).
\end{equation}
By the Borel-Weil theorem for $\SL_2$,
\begin{equation}\label{4.4.2}
\chi_T\left(\mathcal{L}_{X_i}(\lambda+\mu)\right) =e^{-(\lambda +\mu)}+e^{-(\lambda+\mu)+\alpha_i} +\cdots +e^{-(\lambda+\mu)+(m+n)\alpha_i}.
\end{equation}
Similarly, 
\begin{equation}\label{4.4.3}
\chi_T(\L _{X_i}(\mu))=e^{-\mu} +e^{-\mu+\alpha_i}+\cdots + e^{-\mu+m\alpha_i},
\end{equation}
and 
\begin{equation}\label{4.4.4}
\chi_T[\mathcal{O}_e]=e^o.
\end{equation}
By equations \eqref{4.4.1}- \eqref{4.4.3},
\begin{equation}\label{4.4.5}
e^{-\left(\lambda+\mu\right)}\left[1+e^{\alpha_i}+\cdots + e^{\left(m+n\right)\alpha_i}\right]= a_\lambda e^{-\mu} \left[1+ e^{\alpha_i}+\cdots + e^{m\alpha_i}\right]+ b_\lambda e^{-\mu}.
\end{equation}
Take $a_\lambda=e^{-s_i\lambda}=e^{-\lambda+n\alpha_i}$ and 
$$
b_\lambda = \frac{e^{-\lambda}-e^{-s_i\lambda}}{1-e^{\alpha_i}} = e^{-\lambda}\left(\frac{1-e^{n\alpha_i}}{1-e^{\alpha_i}}\right).
$$
Then, considering the two cases $n>0$ and $n\leq0$ separately, it is easy to see that with the above choices of $a_\lambda$ and $b_\lambda$, the equation \eqref{4.4.5} is satisfied for all $\mu$ chosen as above. This proves the proposition. 
\end{proof}

The following corollary follows immediately from Proposition \ref{prop4.4}.

\begin{corollary} \label{4.5}
For  $b_0 \in K_0^\P (*) = K_0^G (*)$,
$$
\left(\mu_i\right)_! \left( \mathcal{O}_{X_i} \boxtimes^\B b_0\right) =b_0 \left[\mathcal{O}_{X_i}\right] \in K_0^\B\left(X_i\right).
$$

Thus, following the proof of Proposition \ref{4.3}, we get that for any $a\in K_0^\B\left(\Y \right)$ and $b\in K_0^\B\left(\X\right)$,
$$
a\odot' \left(b_0\cdot b\right) = b_0a\odot' b. \qquad\qquad\qquad \square
$$
\end{corollary} 

We write the product in $R\left(\B\right)=R(T)$ additively by writing the character $\lambda$ of $\B$ as $e^\lambda$. 
\begin{definition}\label{4.6}
\rm{
 Let $\left\{\omega_i\right\}_{1\leq i\leq l}$ be the fundamental weights of $G$. 
Since $X:=G/B$ is smooth, we have 
$$
K_G^0 \left(X\right)\simeq K_0^G \left(X\right). 
$$
By [\cite{CG}, \S5.2.16],
\begin{equation}\label{4.6.1}
K_0^G\left(X\right) \simeq K_0^G \left( G\times^B *\right) \simeq K_0^B(*) \simeq R(B) \simeq R(T).
\end{equation}
The isomorphism $R(B)\xrightarrow{\sim} K_0^G (X)$ can explicitly be given as 
\begin{equation}\label{4.6.2}
e^\lambda \mapsto \left[\L\left(-\lambda\right)\right],\mbox{ for a character $e^\lambda$ of $T$} ,
\end{equation}
 where $\L\left(-\lambda\right)$ is the homogeneous line bundle over $X$ associated to the principal $B$-bundle $G\rightarrow X$ via the character $e^\lambda$. 
 
By Steinberg [\cite{St}, Theorem 2.2], $R(T)$ is a free $R(G)= R(T)^W$-module (under multiplication) with a basis
$$
\left\{ e^{\delta_x} := x^{-1} \Pi_{\alpha_i: x^{-1} \alpha_i <0} \,\,e^{\omega_i}\right\}_{x\in W}.
$$
Thus, $\left\{\L\left(-\delta_x\right)\right\}_{x\in W}$ is a basis of $K_0^G (X)$
as a $K_0^G (*)\simeq R(T) ^W$-module. 

The above identification \eqref{4.6.1}  easily translates to the identification:
$$
R(T)\simeq R\left(\B\right) \simeq K_0^\P \left(\P/\B\right) , \quad e^\lambda \mapsto \L\left(-\lambda\right)
$$
thought of as a $\P$-equivariant line bundle over $\P/\B$
corresponding to the character $e^{\lambda}$ of $\B$ (equivalently a character of $T$). 

By an analogue of Theorem \ref{2.4} for $X$, we get that the pairing 
$$
\langle~,~\rangle :~ K_G^0 \left(X\right) \otimes_{K_G^0 \left(*\right) } K_G^0 \left(X\right)  \rightarrow K_G^0 \left(*\right) \simeq R(T)^W
$$
induced by 
$$
\langle V_1,~V_2\rangle= \chi_G\left(V_1\otimes V_2\right),
$$
for $G$-equivariant vector bundles $V_1$ and $V_2$ (over $X$)  is non-singular, where $\chi_G$ denotes the $G$-equivariant Euler-Poincar\'e characteristic. 

Let $\left\{\L_x:=\L\left(-\delta_x\right)\right\}_{x\in W}$ be the Steinberg basis of $K_0^\P\left(\P/\B\right)\simeq K_\P^0\left(\P/\B\right)$ (since $\P/\B \simeq X$ is smooth) over $K_0^\P (*)$ and 
let $\left\{\L^x\right\}_{x\in W}$ be the dual basis of $K_0^\P\left(\P/\B\right)$ under the above pairing.}
\end{definition}

Let $\Delta \in K_0^\P \left(\P/\B \times \P/\B\right)$ be the diagonal class, i.e., $\Delta$ is the class of the coherent sheaf $\mathcal{O}_D$, where $D\subset \P/\B \times \P/\B $ is the diagonal variety.

\begin{lemma} \label{4.7}
With the notation as above
$$
\Delta =\sum_{x\in W}\L_x \boxtimes \L^x \in K_0^\P \left(\P/\B \times \P/\B\right). 
$$
\end{lemma}
\begin{proof}
Take any $\P$-homogeneous line bundles $\L\left(\lambda\right)$ and $\L\left(\mu\right)$ over $X=\P/\B$. Then, 
\begin{equation}\label{4.7.1}
\chi_{\P}\left(\mathcal{O}_D\otimes \left(\L\left(\lambda\right)\boxtimes \L\left(\mu\right)\right)\right) = \chi_\P \left(\L\left(\lambda+\mu\right)\right).
\end{equation}

Further, 
\begin{eqnarray}\label{4.7.2}
&&\chi_{\P}\left(\sum_{x\in W}\left(\L_x \boxtimes \L^x \right) \otimes \left(\L\left(\lambda\right)\boxtimes \L\left(\mu\right)\right)\right)\nonumber\\
&&\qquad\qquad\qquad=\sum_{x\in W}\chi_{\P}\left(\left(\L_x \otimes\L(\lambda)\right)\boxtimes \left(\L^x\otimes\L(\mu)\right)\right)\nonumber\\
&&\qquad\qquad\qquad=\sum_{x\in W}\chi_{\P}\left(\L_x \otimes\L(\lambda)\right) \cdot \chi_{\P} \left(\L^x\otimes\L(\mu)\right)\nonumber\\
&&\qquad\qquad\qquad=\sum_{x\in W}\left\langle \L_x ,\L(\lambda)\right\rangle\left\langle \L^x,\L(\mu)\right\rangle \nonumber\\
&&\qquad\qquad\qquad=\left\langle  \sum_{x\in W} \left\langle \L_x ,\L(\lambda)\right\rangle \L^x,\L(\mu)\right\rangle \nonumber\\
&&\qquad\qquad\qquad=\left\langle \L\left(\lambda\right),\L\left(\mu\right)\right\rangle\nonumber\\
&&\qquad\qquad\qquad=\chi_{\P}\left(\L(\lambda +\mu)\right).
\end{eqnarray}
Comparing the equations \eqref{4.7.1} and \eqref{4.7.2}, we get 
$$
\Delta =\left[\mathcal{O} _D\right] = \sum_{w\in W}\L_x \boxtimes \L^x,
$$
since $\left\{ \L\left(\lambda\right)\right\}_{\lambda \in R\left(\B\right)}$ spans $K_0^{\P}\left(\P/\B\right)$.
\end{proof}

\begin{definition}\label{4.8} \rm{
Consider the commutative diagram:
$$
\arraycolsep=2pt
\begin{array}{lcl}
\displaystyle K_0^\B\left(\X\right)&\displaystyle \overset{\overset{i}{\sim}}{\longrightarrow} K_0^\P \left(\P \times^\B \X\right) &\displaystyle \underset{\sim}{\overset{\eta}{\longrightarrow}} K_0^{\P} \left(\P/\B \times \X\right)\\
\rotatebox[origin=c]{140}{$\underset{\sim}{\overset{\underline{\phi}}{\longleftarrow}}$} 
&& \rotatebox[origin=c]{240}{$\underset{\sim}{\overset{{\phi}}{\longrightarrow}}$}  
\\
&\displaystyle  K_0^\P \left(\P/\B\right)  \underset{K_0^\P(*)}{\boxtimes} K_0^\P\left(\X\right).&
\end{array}
$$
In this diagram $i$ is the Induction Isomorphism [CG, \S5.2.15], the isomorphism $\eta$ is induced from the $\P$-equivariant isomorphism of the ind-varieties:
$$
\P \times^{\B} \X \underset{\sim}{\rightarrow} \P/\B \times  \X, \quad \left[ p,x\right] \mapsto\left(p\B,px\right),
\qquad \mbox{for } p\in \P \mbox{ and } x\in \X.
$$

The isomorphism $\phi$ is the Kunneth isomorphism (cf. [CG, Theorem 5.6.1]). To satisfy the hypotheses of loc cit., we have used Lemma \ref{4.7} and the result that 
$$ 
K_0^\P \left(\Y\right) =K_0^G \left(\Y\right),\quad\mbox{ for any } \mbox{ $\P$-ind-variety }\Y. 
$$

By definition, $\overline{\phi}=\phi\circ \eta \circ i$ and hence it is an isomorphism. Analyzing the proof of [CG, Theorem 5.6.1], specifically on page 275 of loc cit., we get that 
\begin{equation}\label{4.8.1}
\overline{\phi} (b)=\sum_{x\in W} \L^x \boxtimes \overline{\mu}_! \left( \L_x \overset{\B}{\boxtimes} b\right), \mbox{ for any } b\in K_0^\B \left(\X\right),
\end{equation}
where $\overline{\mu}:\P\times^\B \X\rightarrow \X$ is the product map $[p,x]\mapsto p\cdot x$, for $p\in\P$ and $x\in \X$. Here, we have  abbreviated $\left(\hat{p}^* \L_x \right) \overset{\B}{\boxtimes} b$
by  $\L_x\overset{\B}{\boxtimes} b$, where $\hat{p}:\P\rightarrow \P/\B$ is the projection.
In particular,  for $b=\mathcal{O}_{X_u}$ (for $u\in \W'$) ,
\begin{equation}\label{4.8.2}
\overline{\phi} \left(\left[ \mathcal{O}_{X_u}\right]\right) = \sum_{x\in W} \L^x \boxtimes \overline{\mu}_! \left( \L_x \overset{\B}{\boxtimes}  \mathcal{O}_{X_u}\right). 
\end{equation}
As mentioned earlier, $\odot'$ is \textit{not} $R(\B)$-linear in the second variable. To remedy this, we modify its definition following [Ka-2, \S8].
Define the \textit{modified convolution product}: 
$$
\odot :K_0^\B \left( \Y\right) \bigotimes _{K_0^\B(*)} K_0^\B \left( \X\right) \rightarrow K_0^\B \left ( \X\right)
$$
by 
\begin{equation}\label{4.8.3}
a\odot b := \sum_{x\in W} \left(\epsilon \left(\L^x\right) \cdot a\right) \, \odot' \overline{\mu}_! \left(\L_x \overset{\B}{\boxtimes} b\right), \mbox{ for  } a\in K_0^\B\left(\Y\right) \mbox{ and } b \in K_0^\B\left(\X\right),
\end{equation}
where $\epsilon:K_0^\P \left(\P/\B\right)\overset{\sim}{\rightarrow}K_0^\B (*)$ is the isomorphism $i^{-1}$ as earlier  for $\X$ replaced by $*$. It is easy to see that $\odot$ does not depend on the choice of the basis $\L_w$. From the definition of $\odot$, it follows that $\odot$ is $K_0^\B(*)$-bilinear. It is clearly $K_0^\B(*)$-linear in the first variable. To prove its linearity in the second variable, take a character $e^\lambda$ of $\B$. Then, 
\begin{eqnarray*}
a\odot e^\lambda \cdot b &=& \sum_{x\in W}\epsilon \left(\L^x\right)\cdot a \odot' \overline{\mu}_!\left(\L_x \overset{\B}{\boxtimes} e^\lambda \cdot b\right)\\
&=& \sum_{x\in W}\epsilon \left(\L^x\right)\cdot a \odot' \overline{\mu}_!\left(\L(-\lambda)\cdot \L_x \overset{\B}{\boxtimes} b\right)\\
&=& \sum_{x,y\in W}\epsilon \left(\L^x\right) \left\langle \L(-\lambda)\cdot \L_x,\L^y\right\rangle\cdot a \odot' \overline{\mu}_! \left( \L_y \overset{\B}{\boxtimes} b\right)\\
&&\text{since $\odot'$ is $R(\P)$-linear in the second variable by Corollary \ref{4.5}}\\
&=&\sum_{y\in W}\left( \sum_{x\in W} \epsilon \left(\L^x\right) \left\langle \L_x, \L\left(-\lambda\right)\L^y \right\rangle\right) \cdot a \odot' \overline{\mu}_! \left( \L_y \overset{\B}{\boxtimes} b\right)\\
&=&\sum_{y\in W} \epsilon\left(\L\left(-\lambda\right)\L^y\right)\cdot a \odot' \overline{\mu}_! \left( \L_y \overset{\B}{\boxtimes} b\right) ,\\
&&\quad\mbox{since }\sum_{x\in W}\L^x \left\langle \L_x, \L\left(-\lambda\right)\L^y \right\rangle = \L\left(-\lambda\right)\L^y\\
&=&\sum_{y\in W}e^\lambda \cdot \epsilon \left(\L^y\right) \cdot a \odot'  \overline{\mu}_! \left( \L_y \overset{\B}{\boxtimes} b\right)\\
&=&e^\lambda a \odot b.
\end{eqnarray*}
This proves that $\odot$ is $K_0^\B(*)$-linear in the second variable. 

We now prove that for $b\in K_0^\P(\X)$, 
\begin{equation}\label{4.8.4}
a\odot b= a\odot' b, \quad \mbox{for any } a\in K_0^\B \left(\Y\right).
\end{equation}
Since $b\in K_0^\P(\X)$, it is easy to see that, for any $x\in W$, 
$$
 \overline{\mu}_! \left( \L_x\overset{\B}{\boxtimes} b\right)= \chi_\P\left(\L_x\right)\cdot b,\,\,\,\text{by the projection formula,}
 $$ since $\L_x\overset{\B}{\boxtimes} b=\tilde{\pi}^* \left(\L_x\right)\otimes \left(\epsilon \overset{\B}{\boxtimes} b\right)$, where $\tilde{\pi}: \P\times^\B \X \rightarrow \P/\B$ is the projection. 
 
 Thus, 
\begin{eqnarray*}
a\odot b &=&  \sum_{x\in W} \epsilon \left(\L^x\right) \cdot a \odot' \chi_\P\left(\L_x\right)\cdot b\\
&=& \sum_{x\in W} \epsilon \left(\L^x\right) \chi_\P\left(\L_x\right)\cdot a \odot' b,\\
&& \mbox{since $\odot'$ is $K_0^\P(*)=R(\P)$-linear in the second variable }\\
&=&\epsilon \left(\mathcal{O}_{\P/\B}\right)\cdot a \odot' b,\mbox{ as above since } \,\chi_\P(\L_x) :=\langle \L_x, \mathcal{O}_{\P/\B}\rangle\\
&=&a \odot' b.
\end{eqnarray*}
This proves \eqref{4.8.4}.

Let $*$ be the Pontryagin product in $K_0^T(\X)$ as in Definition 3.1 and $\odot$ the modified convolution product $K_0^T(\Y)\otimes K_0^T(\X)\rightarrow K_0^T(\X)$ as above. Since $p:\Y \rightarrow \X$ is a $\G$-equivariant (in particular, $\B$-equivariant) fibration; in particular, it is a flat morphism. Thus, there is the pull-back map $p^*:K_0^T(\X) \rightarrow K_0^T(\Y)$. This takes, for $w\in \W'$, $[\mathcal{O}_{X_w}]\mapsto[\mathcal{O}_{X^\B_{xw_o}}]$, where $w_o$ is the longest element of $W$. Via this $p^*$, we get a (modified) convolution product $\odot$ on $K_0^T(\X)$. }
\end{definition}

The following result is due to Kato with  a proof indicated  in [Ka-2, \S8] and [Ka-1, $\S$2.2]. 
\begin{theorem}\label{4.9}
 The two products $*$ and $\odot$ in $K_0^T(\X)$ coincide.  

For any $u,v\in \W'$, write 
\begin{equation}\label{4.9.1}
\left[\mathcal{O}_{X_u}\right] \odot \left[\mathcal{O}_{X_v}\right] = \sum_{w\in \W'} p_{u,v}^w \left[ \mathcal{O}_{X_w}\right].
\end{equation}

Thus,
$$
p_{u,v}^w = b^w_{u,v}, \mbox{ for any } u,v,w\in \W',
$$
where $b_{u,v}^w$ are the structure constants for the Pontryagin product in $K_0^T(\X)$ (cf. identity \eqref{3.1.1}). $\square$
\end{theorem}

Thus, Conjecture 3.3 can equivalently be reformulated in terms of the structure constants for the modified convolution product $\odot$ in $K_0^T(\X)$. 
\begin{conjecture}\label{4.10}
With the above notation, for any $u,v,w\in \W'$ ,
$$
(-1)^{\ell(u)+\ell(v)-\ell(w)}p_{u,v}^w \in \mathbb{Z}_+ \left[\left (e^{\alpha_{1}}-1\right),\ldots , \left (e^{\alpha_{l}}-1\right)\right].
$$
\end{conjecture}
As a corollary of Theorem \ref{4.9} we get the following.

\begin{corollary}\label{4.11}
The product $\odot$ in $K_0^T(\X)$ is associative and commutative. 
\end{corollary}

\begin{proof}
The corollary follows from the corresponding properties of the Pontryagin product $*$ in $K_0^T(\X)$. The associativity of $*$ of course follows since the product $m:\X \times \X \rightarrow \X$ (cf. Definition 3.1) is associative. 

For the commutativity of $*$, recall that the inclusion $\Omega(K)\rightarrow \Omega ^{\cont}(K)$ is $T_0$-equivariantly homotopic equivalence, where $\Omega ^{\cont}(K)$ is the space of all the based continuous maps from $S^1$ to $K$ under the compact-open topology (cf. [PS, Proposition 8.6.6]). Further, $ \Omega ^{\cont}(K)$ being the loop group of a compact Lie group, the coproduct in $K_T^{\top}\left( \Omega ^{\cont}(K)\right)$ is co-commutative and hence the (dual) Pontryagin product $*$ in $K_0^T(\X)$ is commutative. 
\end{proof}

\section{An expression for the modified convolution product structure constants}

We will use Theorem \ref{4.9} to get the structure constants for the Pontryagin product $*$ in $K^T_0(\X)$ from that of the structure constants for the modified convolution product $\odot$ in $K_0^T(\X)$. 

\begin{definition} \label{def5.1} \rm{
Following [KK, \S2.1] consider the ring $Q_\W$ , which is the smash product of the $\W$-field $Q(T)$ ($Q(T)$ being the quotient field of the representation ring $R(T)$) with the group algebra $\mathbbm{Z}[\W]$. Specifically, $Q_\W$ is a free left $Q(T)$-module with basis $\{\delta_w\}_{w\in \W}$ and the product is given by 
$$
\left(q_1\delta_{w_1}\right) \cdot \left( q_2\delta_{w_2}\right) = q_1 \left(w_1\cdot q_2\right) \delta_{w_1w_2}, \quad \mbox{ for } q_1,q_2 \in Q(T) \,\,\text{and  $w_1, w_2 \in \W$},
$$
where $s_0$ acts on $R(T)$ via $s_\theta$. }
\end{definition}

For any simple reflection $\{s_i\}_{0\leq i\leq l}$, define the element $z_i \in Q_\W$ by 
$$
z_i := \frac{1}{1-e^{\alpha_i}} \left( \delta_e-\delta_{s_i}\right),\mbox{ where we take } \alpha_0 = -\theta.
$$
Then, $z_i= e^{-\hat{\rho}}\hat{y}_i\cdot e^{\hat{\rho}}$, where $\hat{y}_i$ is the same as ${y}_i$ in [KK, \S2.1] except that we replace each simple root $\alpha_i$ by $-\alpha_i$ and  $\hat{\rho}(\alpha_i^\vee)=1$ for all simple coroots $\alpha_i^\vee,0\leq i\leq l$ , where 
$\alpha^\vee_0 :=-\theta^\vee$. For any $w\in \W$, define
 $$z_w:=z_{i_1}\cdots z_{i_n}\in Q_\W\,\,\,\text{ for a reduced decomposition $w=s_{i_1}\cdots s_{i_n}\in \W$}.$$
 Then, it does not depend upon the choice of a reduced decomposition of $w$ (i.e.,  $z_{i}$'s satisfy the braid property, cf. [KK, Proposition 2.4]). Moreover,
$$
z_i^2 =z_i \quad \mbox{ for all } 0\leq i \leq l.
$$
Further, we can write (cf. [KK, Theorem 2.9]), for any $w\in \W$,
$$z_w\cdot \left(e^\lambda \delta_e\right) = \sum_{v\leq w} f\left( v,w; \lambda\right) z_v,$$
 for some unique 
$f\left( v,w; \lambda\right) \in R(T)$.
As in [KK, I$_3$], $Q_\W$ acts on $Q(T)$ via
$$
\left(q~\delta_w\right)\boxdot q' = q\cdot \left(wq'\right), \mbox{ for } q,q'\in Q (T) \mbox{ and }w\in \W.
$$
In particular,
\begin{equation} \label{5.1.1} 
z_i\cdot (e^\lambda\delta_e) = e^{s_i\lambda} z_i+ z_i \boxdot e^\lambda, \,\,\,\text{for any $0\leq i\leq l$ and $e^\lambda \in R(T).$}
\end{equation}

As a consequence of Proposition 4.4, we get the following. 
\begin{proposition}\label{5.2}
For any $w\in \W$ and any character $e^\lambda$ of $\B$, 
$$
\mu_!^\B \left(\mathcal{O}_{X_w^\B} \overset{\B}{\boxtimes} e^\lambda\right) = \sum_{v\leq w\in \W} f\left(v,w; \lambda\right)\left[\mathcal{O}_{X_v^\B}\right] \in K_0^\B \left(\Y\right),
$$
where $\mu^\B: \G\times^\B *\rightarrow \Y:=\G/\B$ takes $[g,*]\mapsto g\B$, for $g\in\G$ and, as earlier, $X_w^\B := \overline{\B w\B/\B}\subset \Y$.
\end{proposition}

\begin{proof}
Observe that $\mu_!^\B\left(\mathcal{O}_{X_w^\B}\overset{\B}{\boxtimes}e^\lambda\right)=\L^\B_w(-\lambda)$, 
where $ \L^\B_w(-\lambda)$ is the restriction to $X_w^\B$ of the $\G$-equivariant 
 line bundle over $\Y$ corresponding to the character $e^\lambda$ of $\B$. 
 By Proposition \ref{prop4.4} and the identity \eqref{5.1.1}, the proposition is true for $w=s_i$, for any $0\leq i \leq l$. We assume the validity of the proposition  for $w$ by induction on $\ell(w)$ and take $s_i w\in\W$ with $\ell(s_i w)>\ell(w)$. 
Consider the map 
$$
\mu_i^w: \P_i\times^\B X_w^\B \to X_{s_i w}^\B, \,\,\,\left[p,x\right] \mapsto px, \mbox{ for } p\in \P_i \mbox{ and } x\in X_w^\B.
$$
Then, By [Ku-1, Theorem 8.2.2(c)], $\mu_i^w$ is a trivial morphism, i.e., 
\begin{eqnarray*}
\left(R^j \mu_{i_*}^w\right) \left( \mathcal{O}_{\P_i\times^\B X^\B_w}\right) &=&0, \quad \mbox{for } j>0, \\
&=& \mathcal{O}_{X^\B_{s_iw}}, \quad \mbox{for } j=0.
\end{eqnarray*}
Thus, 
\begin{eqnarray*}
\L^\B_{s_i w}\left(-\lambda\right)&=&\mathcal{O}_{X_i} \overset{\B}{\boxtimes} \L_w\left(-\lambda\right), \mbox{ where } X_i:=\P_i/\B\\
&=& \mathcal{O}_{X_i} \overset{\B}{\boxtimes} \left(\sum_{v\leq w} f\left(v, w; \lambda\right) \left[ \mathcal{O}_{X_v^\B}\right]\right),\mbox{ by induction}\\
&=& \sum_{v\leq w, ~s_iv>v}s_i \left(f\left(v,w;\lambda\right)\right) \left[ \mathcal{O}_{X_{s_iv}^\B}\right]
+  \sum_{v\leq w, ~s_iv<v}s_i \left(f\left(v,w;\lambda\right)\right) \left[ \mathcal{O}_{X_{v}^\B}\right]\\
&&+  \sum_{v\leq w} \left( z_i \boxdot f\left(v,w;\lambda\right)\right) \cdot \left[ \mathcal{O}_{X_{v}^\B}\right], \mbox{ by Proposition \ref{prop4.4}}\\
&=&\sum_{v\leq s_i w}f\left(v,s_iw;\lambda\right)\left[ \mathcal{O}_{X_{v}^\B}\right].
\end{eqnarray*}
The last equality follows since 
\begin{equation}\label{5.2.1}
z_i\cdot \left(aq\right) = \left(z_i\boxdot a\right) q+ \left(s_ia\right)\left(z_i\cdot q\right), \mbox{ for } a\in Q(T) \mbox{ and } q\in Q_W. 
\end{equation}
This completes the induction and hence the proposition is proved. 
\end{proof}

\begin{corollary}\label{5.3}
For any $w\in \W'$, and any character $e^\lambda$ of $\B$, 
$$
\mu_!^\B \left(p^*(\mathcal{O}_{X_w}) \overset{\B}{\boxtimes} e^\lambda\right) = \sum_{v\leq ww_o, \,v\in \W}f\left(v,ww_o; \lambda\right)\left[ \mathcal{O}_{X_{v}^\B}\right],
\,\,\mbox{where $p:\Y\to \X$ is the projection}.
$$
\end{corollary}

\begin{proof}
It follows immediately from Proposition \ref{5.2} and the fact that, under the projection $p: \Y \to \X$, $p^{-1}\left(X_w\right) = X^\B_{ww_o}$, and the discussion before Theorem \ref{4.9}. 
\end{proof}

\begin{definition} \label{5.4} \rm{
For any $0\leq i\leq l$, define a variant of \textit{Demazure operator }
$D_i: R(T) \rightarrow R(T)$ by 
$$
D_i(e^\lambda)=z_i\boxdot e^\lambda = \frac{e^\lambda-e^{s_i\lambda}}{1-e^{\alpha_i}}, \quad \mbox{for any character } e^\lambda \mbox{ of } T .
$$

Recall that $s_0=s_\theta$ and $\alpha_0=-\theta$. Observe that 
\begin{equation}\label{5.4.1}
D_i(ab) =(D_ia)\cdot b+(s_ia) \cdot D_i (b) ,\mbox{ for } a,b \in R(T).
\end{equation}}

\end{definition}

\begin{lemma}\label{5.5}
For any $i_j\in \left\{0,1,\ldots , l\right\}$, as elements of $Q_\W$ (cf. Definition \ref{def5.1}), 
$$
\left(z_{i_1}\cdots z_{i_n}\right) \cdot \left( e^\lambda \delta_e\right) = \sum_{1\leq j_1 < j_2 <\cdots< j_p\leq n} \left( D_{i_1} \cdots \hat{\hat{D}}_{i_{j_1}}\cdots \hat{\hat{D}}_{i_{j_p}} \cdots{{D}}_{i_{n}}\right) \left(e^\lambda\right)
z_{i_{j_1}}\cdots z_{i_{j_p}},
$$
where $\hat{\hat{D}}_j$ means to replace the operator $D_j$ by the Weyl group action of $s_j$. 
\end{lemma}

\begin{proof}
For $n=1$, the lemma follows from the identity \eqref{5.1.1}. We prove the lemma by induction assuming it to be true for $n$ and prove  it for $n+1$. So, take $i_0\in \left\{0,1,\ldots, l\right\}$. Then, 
\begin{eqnarray*}
\left(z_{i_0}~z_{i_1}\cdots z_{i_n}\right) \left(e^\lambda\delta_e\right)&=&z_{i_0}\cdot \left(z_{i_1}\cdots z_{i_n}\cdot \left(e^\lambda\delta_e\right)\right)\\
&=& z_{i_0} \cdot \sum_{1\leq j_1 <\cdots< j_p\leq n}  \left( D_{i_1} \cdots \hat{\hat{D}}_{i_{j_1}}\cdots \hat{\hat{D}}_{i_{j_p}} \cdots{{D}}_{i_{n}}\right) \left(e^\lambda\right)\cdot z_{i_{j_1}}\cdots z_{i_{j_p}}\\
&=&\sum_{1\leq j_1<\cdots < j_p \leq n} [ \left( D_{i_0}~D_{i_1} \cdots \hat{\hat{D}}_{i_{j_1}}\cdots \hat{\hat{D}}_{i_{j_p}} \cdots{{D}}_{i_{n}}\right)\left(e^\lambda\right)\cdot 
 \left(z_{i_{j_1}}\cdots z_{i_{j_p}}\right) \\
 &+&\left(\hat{\hat{D}}_{i_0}~D_{i_1} \cdots \hat{\hat{D}}_{i_{j_1}}\cdots \hat{\hat{D}}_{i_{j_p}} \cdots{{D}}_{i_{n}}\right)\left(e^\lambda\right) \cdot
 \left(z_{i_0}\cdot z_{i_{j_1}}\cdots z_{i_{j_p}}\right) ],
\end{eqnarray*}
by the identity \eqref{5.2.1}. This completes the induction and hence proves the lemma. 
\end{proof}

\begin{definition}\label{5.6} \rm{
For any $x\in W$, similar to the sheaf $\xi^w\left(w\in \W'\right)$, define the sheaf 
$$
\zeta^x =\mathcal{O} _{\mathring{X}^x}\left(-\partial \mathring{X} ^x\right),
$$
where $\mathring{X}^x := \overline{B^-xB/B} \subset X:= G/B$ ,
$
\partial \mathring{X}^x = \mathring{X}^x \backslash \left( B^-xB/B\right)$ and $B^-\supset T$ is the opposite Borel subgroup of $G$.
Consider its class $\left[ \zeta^x\right]\in K_0^T (X) = K^0_T(X)$.

Recall (see, e.g., [KK, Theorem 4.4]) the $K_T^0(*)$-algebra isomorphism
$$
\varphi: ~ R(T) \underset{R(G)}{\otimes} R(T) \underset{\sim}{\rightarrow} K^0_T (X),~e^\lambda\otimes e^\mu \mapsto e^\lambda \cdot \L_X\left(-\mu\right), 
$$
for $e^\lambda, e^\mu $ characters of $T$, where $\L_X(-\mu)$ is the line bundle over $X$ associated to the principal B-bundle $G\rightarrow X$ via the character $e^{-\mu}$ of B and the domain of $\varphi$ acquires the $K_T^0(*)=R(T)$-module structure via its multiplication on the first factor.

The isomorphism $\varphi$ allows us to view $\zeta^x$ as an element $\bar{\zeta}^x\in R(T) \underset{R(G)}{\otimes} R(T) $. 

For any element $\alpha=\sum_j a_j \otimes b_j \in R(T) \underset{R(G)}{\otimes} R(T)$, 
we define 
$$
\vert \alpha \vert = \sum_j a_j b_j \in R(T).
$$
Of course, it is well-defined. }
\end{definition}
\begin{lemma}\label{lem5.4}
For any $x\in W$, 
$$
\vert \bar{\zeta}^x\vert = \delta_{x,e}.
$$
\end{lemma}
\begin{proof}
By definition 
\begin{equation}\label{eq5.7.1}
\vert \bar{\zeta}^x\vert = \langle \zeta^x, \mathcal{O}_{\mathring{X}_e}\rangle,
\end{equation}
where $\mathring{X}_v:= \overline{BvB/B}$ and $\langle\,,\,\rangle: K_0^T(X)\otimes K_0^T (X) \rightarrow R(T)$ is defined similarly as in Definition \ref{2.3}, by setting 
$$
\left\langle \left[\mathcal{S}\right],\left[\F\right]\right\rangle = \sum_i (-1)^i \chi_T \left( X, Tor_i^{\mathcal{O}_X} \left(\mathcal{S},\F\right)\right),
$$
for T-equivariant coherent sheaves $\mathcal{S}$ and $\F$ over $X$. 
By [CK, Proposition 3.8], for $ x,y\in W$,
\begin{equation}\label{eq5.7.2}
\langle \zeta^x ,\mathcal{O} _{\mathring{X}_y}\rangle = \delta _{x,y}; \mbox{ in particular, } \langle \zeta ^x ,\mathcal{O} _{\mathring{X}_e}\rangle =\delta _{x,e}.
\end{equation}
Combining the equations \eqref{eq5.7.1} and \eqref{eq5.7.2}, we get the lemma.
\end{proof}

\begin{definition}\label{def5.8}\rm{
For any $0\leq i \leq l$, we define a certain {\it left Demazure operator: }
$$
D_i' : R(T) \otimes_{R(G)} R(T) \rightarrow R(T)\otimes_{R(G)} R(T),
\,\,\,
D_i'\left( a\otimes b\right)=\left(D_i a\right)\otimes b, \,\text{for $a,b\in R(T)$}. 
$$
Since
$D_i\left(ab\right)=\left(D_i a\right)b$, for $a\in R(T)$ and $b\in R(G)$, 
$D_i'$ is well-defined. 

A slight variant of these operators also appear in [MNS, \S 5.2]. }
\end{definition}

The following is one of our main results of the paper obtained by using Propositions \ref{4.3} and \ref{5.2} and Lemma \ref{5.5}. 
\begin{theorem} \label{5.9}
Take $u\in \W$, $v\in \W'$ and take a reduced decomposition $u=s_{i_1}\ldots s_{i_n}$\,  $\left(0\leq i_j\leq l\right)$. Then, under the modified convolution product (cf. Definition \ref{4.8})
$$
\left[ \mathcal{O}_{X^\B_u}\right] \odot \left[ \mathcal{O}_{X_v}\right] = \sum_{x\in W} \sum_{1\leq j_1< \cdots <j_p\leq n} \left\vert D_{i_1}' \cdots \hat{\hat{D}}'_{i_{j_1}}\cdots \hat{\hat{D}}'_{i_{j_p}} \cdots D_{i_n}'\left(\bar{\zeta}^x \right)\right\vert [ \mathcal{O}_{X_{\overline{s_{i_{j_1}}*\cdots *s_{i_{j_p}}*x*v}}}],
$$
where $\hat{\hat{D}}'_j$ means to replace the operator $D_j'$ by the Weyl
group action on $R(T) \otimes_{R(G)} R(T)$ acting only on the first factor, $*$ is the Demazure product in $\W$ (cf. Definition \ref{4.2}) and for $w\in \W$, $\bar{w}$ denotes the corresponding minimal representative in $wW$. 
\end{theorem}
\begin{proof}
By the identity \eqref{eq5.7.2}, we get for any line bundle $\L$ over $X$,
\begin{equation}\label{5.9.1}
\left[\L \right] =\sum_{x\in W} \left\langle \L,\zeta^x\right\rangle \left[ \mathcal{O}_{\overset{o}{X}_x}\right] \in K^T_0 \left(X\right).
\end{equation}
Write 
\begin{equation}\label{5.9.2}
\left[\zeta^x\right] =\sum_{y\in W}a^x_y\left[\L^y \right] \in K_0^T \left(X\right), \mbox{ for } a^x_y \in R(T),
\end{equation}
where $[\L^y]$ is as defined in Definition \ref{4.6}.
Thus, by the definition of the product $\odot$ (cf. \eqref{4.8.3}),
\begin{eqnarray*}
\left[ \mathcal{O}_{X_u^\B} \right] \odot \left[ \mathcal{O} _{X_v}\right] &=& \sum_{x,y\in W} \epsilon \left(\L^y\right) \left[\mathcal{O}_{X_u^\B}\right] \odot' \left\langle \L_y, \zeta^x\right\rangle \overline{\mu}_! \left( \mathcal{O}_{\mathring{X}_x} \overset{\B}{\boxtimes} \mathcal{O}_{X_v} \right), \mbox{ by } \eqref{5.9.1}\\
&=& \sum_{x,y\in W} \epsilon \left(\L^y\right) \left[\mathcal{O}_{X^B_u}\right] \odot' a_y^x \left[ \mathcal{O}_{X_{\overline{x*v}}}\right], \mbox{ by } \eqref{5.9.2} \mbox{ and}\\&&\mbox{ Proposition \ref{4.3}
 since }\mathring{X}_x =X^\B_x\\
&=& \sum_{1\leq j_1 < \cdots < j_p \leq n} \,\,\sum_{x,y\in W} \epsilon \left(\L^y\right) \left( D_{i_1} \ldots {\hat{\hat{D}}}_{i_{j_1}} \ldots {\hat{\hat{D}}}_{i_{j_p}} \ldots D_{i_n}\right) \left( a^x_y
\right)\\
&&\cdot [ \mathcal{O}_{X_{s_{i_{j_1}}*\cdots *s_{i_{j_p}}}}] \odot' \left[ \mathcal{O}_{\overline{x*v}}\right],  \mbox{ by Proposition \ref{5.2} and  Lemma \ref{5.5}}\\
&=& \sum_{1\leq j_1 < \cdots < j_p \leq n} \,\,\sum_{x\in W} \left[\sum_{y\in W} \epsilon \left(\L^y\right) \left( D_{i_1} \ldots \hat{\hat{D}}_{i_{j_1}} \ldots \hat{\hat{D}}_{i_{j_p}} \ldots D_{i_n}\right)
 \left( a^x_y\right)\right]\\
&&[ \mathcal{O}_{X_{\overline{s_{i_{j_1}}*\cdots *s_{i_{j_p}}*x*v}}}]\mbox{, by Proposition \ref{4.3}}\\
&=& \sum_{x\in W} \,\,\sum_{1\leq j_1 < \cdots < j_p \leq n}  \left\vert D_{i_1}' \ldots \hat{\hat{D}}'_{i_{j_1}} \ldots \hat{\hat{D}}'_{i_{j_p}} \ldots D_{i_n}' \left(\bar{\zeta}^x\right)\right\vert
\cdot [ \mathcal{O}_{X_{\overline{s_{i_{j_1}}*\cdots *s_{i_{j_p}}*x*v}}}], \mbox{ by \eqref{5.9.2}}.
\end{eqnarray*}
This proves the Theorem. 
\end{proof}

\begin{remark}\rm{
For $u,v\in \W'$, as in Definition \ref{4.8}, 
$$
\left(p^*\left[ \mathcal{O}_{X_u}\right]\right) \odot \left[\mathcal{O}_{X_v}\right] =\left[\mathcal{O}_{X^\B_{uw_o}}\right] \odot \left[\mathcal{O}_{X_v}\right].
$$

Thus, the above Theorem \ref{5.9} gives an expression for $\left[ \mathcal{O}_{X_u}\right] \odot \left[\mathcal{O}_{X_v}\right]$  for the (modified) convolution product $\odot$ in $K^T_0(\X)$ replacing $u$ by $uw_o$. 

Recall the isomorphism $\varphi: R(T) \otimes_{R(G)}R(T)\overset{\sim}{\rightarrow} K^0_T(X)$ from Definition \ref{5.6}.} 
\end{remark}

\begin{lemma}\label{5.11}
For $x\in W$ and $1\leq i \leq l$,

(a) \begin{eqnarray*} s_i'\left\{\left[\zeta^x\right]\right\}&=& \left\{ 
\begin{array}{l}
\displaystyle e^{\alpha_i}\left[\zeta^x\right],\mbox{ if} \quad s_ix>x\\
\displaystyle \left[\zeta^x\right] + \left(1-e^{\alpha_i}\right) \left[\zeta^{s_i x}\right], \mbox{ if}\quad s_ix<x.
\end{array}\right.\\
\end{eqnarray*}

(b) \begin{eqnarray*} D_i'\left[\zeta^x\right] &=& 
\left\{ 
\begin{array}{l}
\displaystyle\left[\zeta^x\right],\mbox{ if} \quad s_ix>x\\
\displaystyle- \left[\zeta^{s_ix}\right] , \mbox{ if}\quad s_ix<x.
\end{array}\right.
\end{eqnarray*}
\end{lemma}

\begin{proof}
(a) Observe first that, for any $y \in W$, 
\begin{equation}\label{5.11.1}
\left\langle s_i'\left[\zeta^x\right],s_i'[ \mathcal{O}_{\mathring{X}_y}]\right\rangle = s_i\left\langle \left[\zeta^x\right],[ \mathcal{O}_{\mathring{X}_y}]\right\rangle=\delta_{x,y}, \mbox{ by identity \eqref{eq5.7.2}}.
\end{equation}
We first take $s_ix>x$. Then, by [MNS, Proposition 5.5], 
\begin{eqnarray*}
\left\langle e^{\alpha_i}\left[\zeta^x\right], s_i'[ \mathcal{O}_{\mathring{X}_y}]\right\rangle =
\left\{ \begin{array}{l}
0,\quad \mbox{if}\quad s_iy<y\\
\delta_{x,y},\quad\mbox{if}\quad s_iy>y.
\end{array}\right.
\end{eqnarray*}

This proves (a) in the case $s_ix>x$ using equation \eqref{5.11.1}. 

Now, take $s_ix<x$. Then, 
\begin{eqnarray*}
\left\langle\left[\zeta^x\right],  \left(1-e^{\alpha_i}\right) \left[\zeta^{s_i x}\right], s_i'[ \mathcal{O}_{\mathring{X}_y}]\right\rangle=
\left\{ \begin{array}{l}
0,\quad \mbox{if}\quad s_iy>y\\
\delta_{x,y},\quad\mbox{if}\quad s_iy<y.
\end{array}\right.
\end{eqnarray*}

This proves (a) in the case $s_i x<x$ again using equation \eqref{5.11.1}. 
\vskip1ex

(b) Since $D'_i=\frac{1}{1-e^{\alpha_i}} \left(\Id-s_i'\right)$, part (a) proves (b).  
\end{proof}

\begin{definition} \label{5.12} \rm{
Define an involution 
$$
t:~R(T)\otimes_{R(G)}R(T) \rightarrow R(T)\otimes_{R(G)}R(T) ,~a\otimes b \mapsto b\otimes a,\,\mbox{for } a,b\in R(T).
$$

Via the isomorphism $\varphi$ of Definition \ref{5.6}, we identify any element of $K_T^0(X)$ by an element of $R(T)\otimes_{R(G)}R(T)$. 
Thus, for any class $\eta\in K_T^0(X)$, we have the transposed class $\eta^t := t(\eta) \in K_T^0 (X)$ under the 
isomorphism $\varphi$. In fact, the same definition as that of $\varphi$ realizes $\eta^t\in  K_T^0\left(\X^{\B}\right)$ compatible with its restriction to $X\hookrightarrow \bar\Y$, where $\bar\Y$ is as in the proof of Lemma \ref{2.7}. Viewed $\eta^t$ as an element of $K_T^0\left(\bar\Y\right)$, we write it as $\eta^t_{\aff}$. 
We record this as an $R(T)$-algebra homomorphism:
$$
K_T^0\left(X\right) \rightarrow  K_T^0\left(\bar\Y\right),\quad \eta \mapsto \eta^t_{\aff}.
$$
In particular,  we have (for any $x\in W$), 
\begin{equation}\label{5.12.1}
\left[\zeta^x\right]^t_{\aff} \in K_T^0\left(\bar\Y\right).
\end{equation}
Let $B\subset P_i~\left(1\leq i \leq l\right)$ be the minimal parabolic subgroup of $G$ containing $s_i$. Consider the projection  $p_i:X\rightarrow G/P_i $. Recall the Demazure operator
$$
\D_i : K_T^0\left(X\right) \rightarrow K_T^0\left(X\right),\quad \eta \mapsto p_i^* \left( \left(p_i\right)_! \eta \right). 
$$
Under the identification $\varphi$, we can think of the Demazure operators acting on $R(T)\otimes_{R(G)}R(T)$. Then, 
$$
\D_i ~\left(a\otimes b\right) =a\otimes \left(\frac{b-\left( s_i b\right)e^{\alpha_i}}{1-e^{\alpha_i}}\right),\quad\text{for $a,b \in R(T)$}.
$$
Thus, 
\begin{equation}\label{5.12.2}
\D_i~\left(a\otimes b\right)= \left( 1\otimes e^\rho\right) \cdot \left( D_i''\left(a\otimes e^{-\rho} b \right)\right),
\end{equation}
where $\rho$ is the half sum of positive roots of $G$
and 
\begin{equation} \label{5.12.3}D_i''\left(a\otimes b\right)=a\otimes \left(D_i b\right)\quad \text{ (cf. Definition \ref{5.4}).}
\end{equation} 

Recall that $\left\{ e_x := e^{\delta_x}\right\}_{x\in W}$ is the Steinberg basis of $R(T)$ over $R(G)$ (cf. Definition \ref{4.6}). Thus, for any $y\in W$, we can write as elements of $R(T) \otimes _{R(G)} R(T)$: 
$$
\left[ \zeta^y\right] = \sum_{x\in W} r^y_x \otimes e_x =\sum_{x\in W} e_x \otimes q_x^y.
$$}
\end{definition} 

\begin{lemma}\label{5.13}
For any $x,y \in W$, 
$$
r_x^y = \sum_{z \in W} \zeta^y  (z) F_{z,x}\quad\mbox{ and } q^y_x = \sum_{z\in W} \left(z \cdot \zeta^y (z^{-1})\right) F_{z,x},
$$
where $F$ is the inverse of the matrix 
$E:= \left(E_{x,y} := y\cdot e_x\right)_{x,y \in W}$ and $\zeta^y(z)\in R(T)$ is the localization of $\zeta^y$ at $z$. By [St, \S2], det $E\neq 0$. 

In particular, as elements of $R(T)\otimes_{R(G)}R(T)$,
\begin{align}\label{5.13.1}
\left[ \zeta^{w_o}\right] ^t &= (-1) ^{\ell(w_o)} \left(1\otimes e^{-2\rho}\right) \cdot \left[ \zeta^{w_o}\right]\nonumber\\
&= (-1)^{\ell(w_o)} \left( e^\rho \otimes e^{-\rho} \right) \cdot \left[ \zeta^{w_o}\right]\nonumber\\
&= (-1)^{\ell(w_o)} \left( e^{2\rho} \otimes 1 \right) \cdot \left[ \zeta^{w_o}\right].
\end{align}
\end{lemma}

\begin{proof}
We fix $y\in W$ and drop the superscript $y$ in the above. Let $\bar{\zeta}$ be the row matrix $\bar{\zeta}_x= {\zeta}^y (x)$, for $x\in W$ and let $\bar{r}$ be the row matrix $\bar{r}_x= r_x^y$. Then, $\bar{r}\cdot E=\bar{\zeta}$. We have used here that the localization of the line bundle $\L(\lambda)$ at $x$ is $e^{-x\lambda}$. 
Let $\bar{q}$ be the row vector $\bar{q}_x=q_x^y$. 
Then, $\bar{q} \cdot E=\hat{\zeta}$, where $\hat{\zeta}$ is the row matrix
$\hat{\zeta}=x\cdot {\zeta}^y (x^{-1})$. Thus, 
$$
\bar{r}=\bar{\zeta}\cdot F \quad \mbox{and}\quad \bar{q} =\hat{\zeta}\cdot F.
$$
This proves the first part of the lemma. 

To prove \eqref{5.13.1}, observe that (cf. [KK, Proposition 2.22])
\begin{equation}\label{5.13.2}
\zeta^{w_o} (x) =
\left\{
\begin{array}{l}
\displaystyle 0, \quad \mbox{for } x<w_o\\
\displaystyle \prod_{\alpha\in R^+ }\left(1-e^{-\alpha}\right),\quad \mbox{for } x=w_o,
\end{array}
\right.
\end{equation}
where $R^+$ is the set of positive roots of $G$. (Note that $e^{w_o,w_o}$ from loc. cit. equals $\overline{\zeta^{w_o}(w_o)}$.) Observe that, by definition, 
\begin{equation}\label{5.13.3}
\zeta^{w_o} =\mathcal{O}_{\left\{w_o \right\}}.
\end{equation}

Now, \eqref{5.13.1} follows from \eqref{5.13.2} and the first part of the lemma. To prove the equality $\left(1\otimes e^{-2\rho}\right)\cdot \left[\zeta^{w_o}\right]=\left(e^\rho \otimes e^{-\rho}\right)\cdot \left[\zeta^{w_o}\right]=\left(e^{2\rho}\otimes 1\right) \left[\zeta^{w_o}\right]$ as in \eqref{5.13.1}, consider their localization since $K_T(X)\rightarrow K_T(X^T)$ is injective (cf. [KK, Theorem 3.13]), and use \eqref{5.13.2}. 
\end{proof}

\begin{proposition}\label{5.14}
For any $x\in W$, take a reduced decomposition $\left( s_{i_1}\ldots s_{i_n}\right)\cdot x=w_o$. Then, as elements of $R(T) \otimes _{R(G)} R(T)$, 
\begin{eqnarray*}
&(a)& \,\,\left[\zeta^x\right]^t = (-1)^{\ell(x)}\left( e^{2\rho}\otimes 1\right) \cdot \left( D_{i_n}''\circ \cdots \circ D_{i_1}'' \left( \left[ \zeta^{w_o}\right]\right)\right)\\
&(b)&\,\,\qquad  = (-1)^{\ell(x)}\left( e^{\rho}\otimes e^{-\rho}\right) \cdot \left( \D_{i_n}\circ \cdots \circ \D_{i_1} \left( \left[ \zeta^{w_o}\right]\right)\right)\\
&(c)&\,\, \qquad= (-1)^{\ell(x)}\left( e^{\rho}\otimes e^{-\rho}\right) \cdot \left[ \mathcal{O}_{\mathring{X}^{x^{-1}}}\right],
\end{eqnarray*}
where $\mathring{X}^{x^{-1}} := \overline{B^- x^{-1} B/B} \subset X$ as before and the operators
$\D_i$ an $D_i''$ are defined in Definition \ref{5.12}. 
\end{proposition}

\begin{proof}
(a) Observe first that for any $\left[\zeta\right]\in K_T^0\left(X\right) $,
\begin{equation}\label{5.14.1}
\left(D_i' \left[\zeta \right] \right)^t = D_i'' \left(\left[\zeta \right] ^t\right).
\end{equation}
By Lemma \ref{5.11} (b), 
$$
\left( D_{i_n}'\circ \cdots \circ D_{i_1}'\right)  \left[ \zeta^{w_o}\right] =(-1)^n \left[ \zeta^{x}\right].
$$
Taking the transpose, we get 
\begin{equation}\label{5.14.2}
\left(\left( D_{i_n}'\circ \cdots \circ D_{i_1}' \right) \left[ \zeta^{w_o}\right] \right)^t = (-1)^n\left[ \zeta^{x}\right]^t.
\end{equation}
By equation \eqref{5.14.1},
$$
\left(\left( D_{i_n}'\circ \cdots \circ D_{i_1}' \right) \left[ \zeta^{w_o}\right] \right)^t = \left( D_{i_n}''\circ \cdots \circ D_{i_1}'' \right) \left(\left[ \zeta^{w_o}\right]^t \right).
$$
 Thus, by equation \eqref{5.14.2}, 
\begin{eqnarray*}
 (-1)^n\left[ \zeta^{x}\right]^t &=& \left( D_{i_n}''\circ \cdots \circ D_{i_1}'' \right) \left(\left[ \zeta^{w_o}\right]^t \right)\\
&=& (-1)^{\ell(w_o)}\left( D_{i_n}''\circ \cdots \circ D_{i_1}'' \right) \left(\left(e^{2\rho}\otimes 1 \right)\cdot \left[ \zeta^{w_o} \right]\right), \mbox{ by \eqref{5.13.1}}\\
&=&(-1)^{\ell(w_o)}\left(e^{2\rho}\otimes 1 \right)\left( D_{i_n}''\circ \cdots \circ D_{i_1}'' \right)\left(\left[ \zeta^{w_o}\right] \right).
\end{eqnarray*}
This proves (a).
\vskip1ex

(b) By \eqref{5.12.2}, 
$$
\D_i \left(\alpha\right) =\left( 1\otimes e^\rho\right) D_i'' \left( \left(1\otimes e^{-\rho}\right)\cdot \alpha\right) , \mbox{ for any } \alpha \in R(T) \otimes _{R(G)} R(T).
$$
Further, $\left( 1\otimes e^\rho\right) \cdot\left[ \zeta^{w_o} \right] =\left(e^{-\rho}\otimes 1 \right)\cdot\left[ \zeta^{w_o} \right] $
(see the proof of \eqref{5.13.1}). Thus, (b) follows from (a). 
\vskip1ex

(c) By the (b)-part, 
\begin{eqnarray*}
\left[ \zeta^{x}\right]^t &=&  (-1)^{\ell(x)}\left( e^{\rho}\otimes e^{-\rho}\right) \cdot \left( \D_{i_n}\circ \cdots \circ \D_{i_1} \left( \left[ \zeta^{w_o}\right]\right)\right)\\
&=&(-1)^{\ell(x)}\left( e^{\rho}\otimes e^{-\rho}\right) \cdot \left( \D_{i_n}\circ \cdots \circ \D_{i_1} \left( [ \mathcal{O}_{\left\{w_o\right\}}]\right)\right)\\
&=& (-1)^{\ell(x)}\left( e^{\rho}\otimes e^{-\rho}\right) \cdot [ \mathcal{O}_{w_o \mathring{X}_{w_ox^{-1}}}], \mbox{ by the definition of } \D_i\\
&=& (-1)^{\ell(x)}\left( e^{\rho}\otimes e^{-\rho}\right) \cdot [ \mathcal{O}_{\mathring{X}^{x^{-1}}}].
\end{eqnarray*}
This proves (c), completing the proof of the proposition.
\end{proof}

For any $u\in \W, v\in \W'$ and $x\in W$, consider 
$$ X_{(u,x,v)} := X_u' \times^{\B}\mathring{X}_x' \times{^\B} X_v $$
 together with the standard product map $\mu_x:X_{(u,x,v)}\rightarrow\X$ and the standard projection $\pi_x: X_{(u,x,v)}\rightarrow X_{u}^\B$, where $X'_u$ is the inverse image of $X_u^\B$ in $\G$ under $\G\rightarrow \Y:=\X^\B$, 
$\mathring{X}_x \subset X\hookrightarrow \Y$ and $\mathring{X}_x'$ is to be thought of as its inverse image in $\G$. 
Observe that $X_{(u,x,v)}$ is a projective variety and  $\pi_x$ is a fibration (in particular, a flat morphism). Thus, the pull-back $\pi_x^*$ is well-defined. Also, we have the standard pull-back map
$$
\mu_x^*:K_T^0\left(\X\right) \rightarrow K_T^0\left( X_{(u,x,v)}\right).
$$
In particular, $\mu_x^* \left(\xi^w\right)$ is well-defined for any $w\in \W'$. Since $ X_{(u,x,v)}$ is a projective variety, both $\mu_{x!}$ and $\pi_{x!}$ are well-defined.

We give another expression for the modified convolution product $\odot$ in the following:
\begin{theorem}\label{5.15}
For $u\in\W$ and $v\in\W'$, 
$$
\left[\mathcal{O}_{X_u^\B}\right]\odot\left[\mathcal{O}_{X_v}\right] =\sum_{w\in \W'}\,\,\sum_{x\in W} \left\langle \left( \left[ \zeta^x\right]^t_{\aff} \right)_{|X_u^\B},\pi_{x!}\mu_x^*\xi^w\right\rangle\left[\mathcal{O}_{X_w}\right],
$$
where $\left[ \zeta^x\right]^t_{\aff}$ is as in \eqref{5.12.1}. 
\end{theorem}

\begin{proof}
By the definition of $\odot$ as in \eqref{4.8.3}, 
\begin{align}\label{5.15.1}
\left[\mathcal{O}_{X_u^\B}\right]\odot\left[\mathcal{O}_{X_v}\right] &= \sum_{y\in W} \epsilon \left(\L^y\right) \left[\mathcal{O}_{X_u^\B}\right]\odot'\overline{\mu}_!\left(\L_y \overset{\B}{\boxtimes} \mathcal{O}_{X_v}\right)\nonumber\\
&=\sum_{x,y\in W} \epsilon \left(\L^y\right) \left[\mathcal{O}_{X_u^\B}\right]\odot'\left\langle \L_y,\zeta^x\right\rangle \overline{\mu}_! \left( \mathcal{O}_{\mathring{X}_x}\overset{\B}{\boxtimes}\mathcal{O}_{X_v}\right) \nonumber\\
&\qquad \qquad \mbox{ by [CK, Proposition 3.8]}\nonumber\\
&= \sum_{x\in W}\, \sum_{y\in W} \epsilon\left(\L^y\right)  \left[\mathcal{O}_{X_u^\B}\right]\odot'\left\langle \L_y,\zeta^x\right\rangle \mathcal{O}_{X_{\overline{x*v}}},
\mbox{  by Proposition \ref{4.3}.}
\end{align}
Write 
$$
\left[\zeta^x\right] = \sum_i a_i \otimes b_i \in R(T) \otimes _{R(G)} R(T).
$$
Then, we have 
$$
\left\langle \L_y,\zeta^x\right\rangle =  \sum_i a_i\left\langle \L_y,\alpha(b_i) \right\rangle,
$$
where $\alpha:R(T)\rightarrow K_G^0(X)$ is the ring isomorphism induced by $e^\lambda\mapsto\L_X(-\lambda)$. 

Thus, the expression \eqref{5.15.1} becomes 
\begin{eqnarray*}
&&\sum_{x\in W}\sum_i\sum_{y\in W} \epsilon  \left(\L^y\right)
\langle \L_y, \alpha(b_i)\rangle  \left[\mathcal{O}_{X_u^\B}\right]\odot' a_i\mathcal{O}_{X_{\overline{x*v}}} \,,\\
&&\mbox{since $\odot'$ is $R(G)$-linear in the second variable }\\
&=&\sum_{x\in W}\sum_i b_i\left[\mathcal{O}_{X_u^\B}\right] \odot' a_i \mathcal{O}_{X_{\overline{x*v}}} \\
&=&\sum_{x\in W} \left( \left(\left[ \zeta^x\right]^t_{\aff}\right)_{|_{X_u^\B}}\right) \odot' \left[\mathcal{O}_{X_{\overline{x*v}}}\right].
\end{eqnarray*}
Thus, from \eqref{5.15.1}, we get 
$$
\left[\mathcal{O}_{X_u^\B}\right]\odot \left[\mathcal{O}_{X_v}\right] = \sum_{x\in W}\left( \left(\left[ \zeta^x\right]^t_{\aff}\right)_{|_{X_u^\B}}\right) \odot' \left[\mathcal{O}_{X_{\overline{x*v}}}\right].
$$
Hence, by Theorem \ref{2.4},
\begin{eqnarray*}
\left[\mathcal{O}_{X_u^\B}\right]\odot \left[\mathcal{O}_{X_v}\right] &=& \sum_{w\in \W, \, x\in W}\langle \left((\left[ \zeta^x\right]^t_{\aff})_{|_{X_u^\B}}\right) \odot' [\mathcal{O}_{X_{\overline{x*v}}}], 
\xi^w\rangle [\mathcal{O}_{X_w}]\\
&=&\sum_{w\in \W,\, x\in W}\, \langle\mu_{x_!}\left(\pi_x^*((\left[ \zeta^x\right]^t_{\aff})_{|_{X_u^\B}})\right),\xi^w \rangle [  \mathcal{O}_{X_w}]
\\
&=& \sum_{w\in \W, \, x\in W}\left\langle \left(\left[ \zeta^x\right]^t_{\aff}\right)_{|X_u^\B}, \pi_{x_!} \mu_x^* \xi^w \right\rangle  [\mathcal{O}_{X_w}],
 \mbox{ by the next lemma.}
\end{eqnarray*}
 This proves the theorem. 
\end{proof}

\begin{lemma}\label{5.16}
For any morphism of projective varieties $f:X\rightarrow Y$ and locally free sheaf $\F$ on $Y$ and a coherent sheaf $\uS$ on $X$, 
$$
\chi\left(\F\otimes f_! \uS\right) = \left\langle \F, f_! \uS\right\rangle = \left\langle f^*\left(\F\right),\uS\right\rangle.
$$
\end{lemma}

\begin{proof}
Consider $\epsilon:Y\rightarrow *$. Then,
\begin{eqnarray*}
\left\langle f^* \left(\F \right),\uS\right\rangle = \chi\left(f^* \left(\F \right) \otimes \uS \right) &=& \epsilon_! f_! \left( f^* \left(\F \right) \otimes \uS\right)\\
&=& \epsilon_! \left( \F\otimes f_! \left(\uS\right)\right), \mbox{ by the Projection Formula}\\
&=& \chi\left(\F\otimes f_! \left(\uS\right) \right).
\end{eqnarray*}
\end{proof}

Let $Q^\vee_{\leq 0}:=\left\{ q\in Q^\vee:\alpha(q)\leq 0,\mbox{ for all the positive roots $\alpha$ of $G$}\right\}$, where $Q^\vee$ is the coroot lattice of $G$.
We write such a $q$ as $q\leq0$. Also, $\tau_q\in \W$ denotes the  translation by $\alpha$. 

\begin{lemma} \label{5.17}
(a) For $q\leq0$, $\tau_q\in \W'$ .

\vskip1ex
(b) For $q\leq0$ and $x\in W'_q$ , $\ell(x\tau_q)=\ell(\tau_q)-\ell(x)$, where $W_q\subset W$ is the stabilizer of $q$ in $W$ and $W_q'$ is the set of smallest coset representatives in $W/W_q$. 
\vskip1ex
(c) For $q\leq0$ and $x\in W'_q$ , $x \tau_q\in \W'$. 
Conversely, any element of $\W'$ can be written as $x\cdot \tau_q$  for some $q\leq0$ and $x\in W'_q$. 
\end{lemma}

\begin{proof}
Recall that for any $q\in Q^\vee$ and $x\in W$, 
 \begin{equation}\label{5.17.1}
 \ell\left(x\tau_q\right) =\left( \sum_{\alpha\in R^+ \cap x^{-1} R^-}\left\vert \alpha (q)+1\right\vert \right) +\sum_{\alpha\in R^+ \cap x^{-1} R^+}\left\vert \alpha (q)\right\vert 
 \end{equation}
 (cf. [IM]), where $R^+$ is the set of positive roots of $G$ and $R^- :=-R^+$. 
\vskip1ex

(a) For $q\leq 0$ and any simple reflection $s_i$, $1\leq i\leq l$, 
\begin{eqnarray*}
\ell\left(\tau_q \cdot s_i\right) &=& \ell\left(s_i\tau_{-q}\right)\\
&=& \left\vert \alpha_i (q) -1 \right\vert - \sum_{\alpha\in R^+ \backslash \left\{ \alpha_i\right\}} \alpha(q) , \mbox{  by \eqref{5.17.1}}\\
&=&1-\sum_{\alpha\in R^+} \alpha(q) \\
&=& \ell(\tau_q)+1.
\end{eqnarray*}
This proves (a).

\vskip1ex
(b)
\begin{equation}\label{5.17.2}
\ell\left(x\tau_q\right) = -\sum_{\alpha\in R^+\cap x^{-1}R^- :\atop \alpha(q) \neq 0} \left( \alpha (q) +1\right) - \sum_{\alpha\in R^+ \cap x^{-1}R^+} \alpha(q) + \sum_{\alpha\in R^+\cap x^{-1}R^- :\atop \alpha(q) = 0}  1,  \mbox{ by \eqref{5.17.1}}.
\end{equation}
We assert that there does not exist any $\alpha\in R^+\cap x^{-1}R^-$ such that $\alpha(q) = 0$. This follows since $\alpha(q) = 0\Leftrightarrow s_\alpha q=q \Leftrightarrow s_\alpha \in W_q$. If $x\alpha \in R^-$, then $x s_\alpha<x$ by [Ku-1, Lemma 1.3.13]. This contradicts the choice of $x\in W_q'$. Thus, by \eqref{5.17.2},
$$
\ell\left(x\tau_q\right)=\ell\left(\tau_q\right)-\ell(x),\mbox{ proving (b).}
$$
\vskip1ex

(c) We first prove that for $q\leq 0$ and $x\in W'_q$, $x \tau_q \in \W'$. 

Take a simple reflection $s_i \in W$. If $x\tau_q s_i< x\tau_q$, then 
$$
\ell\left(x\tau_qs_i\right)=\ell\left(x\tau_q\right)-1=\ell\left(\tau_q\right)-\ell(x) - 1,\mbox{ by (b)} .
$$
This gives 
$$
\ell\left(\tau_qs_i\right)=\ell\left(x^{-1}\cdot x\tau_q s_i\right) \leq \ell(x^{-1})  +\ell\left(x\tau_q x_i\right) = \ell\left(\tau_q\right) -1.
$$
This contradicts (a), proving that $x\tau_q\in \W'$. 

For the converse, take any element $w\in \W'$ and write it as $w=y\cdot \tau_q =y \tau_{x\cdot q'}$, for $x,y\in W$ and $q'\leq0$. 
Thus, 
\begin{align}\label{5.17.3}
&w=yx\tau_{q'} x^{-1} = x_1y_1 \tau_{q'} x^{-1} ,\mbox{ for $x_1\in W'_{q'}$ and $y_1\in W_{q'}$}\nonumber\\
&= x_1 \tau_{q'} y_1x^{-1}.
\end{align}
Since $x_1\tau_{q'}\in \W'$ by the first part of $(c)$, and $w\in\W'$ (by assumption), we get $w=x_1\tau_{q'}$ by \eqref{5.17.3}. This completes the proof of (c). 
\end{proof}

\begin{lemma} \label{5.18}
$$
\sum_{x\in W} \bar{\zeta}^x =1\otimes 1.
$$
\end{lemma}

\begin{proof}
For any $y\in W$, by [CK, Proposition 3.8], 
$$
\langle \sum_{x\in W} \zeta^{x} ,\mathcal{O}_{\mathring{X}_y}\rangle= 1.
$$
Also, 
$$
\langle \varphi\left(1\otimes 1\right),\mathcal{O}_{\mathring{X}_y}\rangle=1, \mbox{ where $\varphi$ is as in Definition \ref{5.6}.} 
$$
Thus, the lemma follows. 
\end{proof}

A slightly weaker version of the following result is obtained by Kato [Ka-1, Theorem 1.7] by a different method, where he assumed that $q<0$. 

\begin{proposition}\label{5.19}
For any $u\in \W$ and $q\in Q^\vee_{\leq0}$, 
\begin{equation}\label{5.19.1}
\left[\mathcal{O}_{X_u^\B}\right]\odot \left[\mathcal{O}_{X_{\tau_q}}\right] = \left[ \mathcal{O}_{X_{\overline{u*\tau_q}}}\right].
\end{equation}
In particular, for $u\in\W'$ and $q\in Q_{\leq0}^\vee$, 
\begin{equation}\label{5.19.2}
\left(p^*\left[\mathcal{O}_{X_u}\right]\right)\odot \left[\mathcal{O}_{X_{\tau_q}}\right] = \left[ \mathcal{O}_{X_{{u\cdot\tau_q}}}\right].
\end{equation}
\end{proposition}

Observe that $u\cdot\tau_q\in\W'$ (as shown in the proof below).

\begin{proof}
Observe first that for any $x\in W$, 
\begin{equation}\label{5.19.3}
\overline{x*\tau_q}= \tau_q.
\end{equation}
To prove this, write $x=x'x''$ with $x'\in W_q'$ and $x''\in W_q$. 
Thus,
\begin{eqnarray*}
\overline{x*\tau_q} =\overline{x'*x''*\tau_q}&=& \overline{x'*\left(\bar{x}\cdot \tau_q\right)},\mbox{ for some $\bar{x}\in W_q$}\\
&=& \overline{x'*\tau_q}\\
&=& \tau_q, \mbox{ by Lemma \ref{5.17}(b).}
\end{eqnarray*}
Take a reduced decomposition $u=s_{i_1}\ldots s_{i_n} \in \W$.
By Theorem \ref{5.9} and identity \eqref{5.19.3}, we get 
\begin{eqnarray*}
\left[\mathcal{O}_{X_u^\B}\right]\odot \left[\mathcal{O}_{X_{\tau_q}}\right] &=& \sum_{1\leq j_1 <\cdots < j_{p}\leq n} \left \vert \left( D_{i_1}'\circ \cdots \circ \hat{\hat{D}}'_{i_{j_1}}\circ \cdots \circ  \hat{\hat{D}}'_{i_{j_p}} \circ \cdots \circ  D'_{i_n}\right) 
 (\sum_{x\in W} \bar\zeta^x)\right\vert \cdot \\
 && [\mathcal{O}_{X_{\overline{s_{i_{j_1}}*\cdots * s_{i_{j_p}}*\tau_q}}}]\\
&=& [\mathcal{O}_{X_{\overline{u*\tau_q}}}], \mbox{ by Lemma \ref{5.18}.}
\end{eqnarray*}
This proves \eqref{5.19.1}. 
By  \eqref{5.19.1}, we get for $u\in \W'$, 
\begin{eqnarray*}\label{5.19.4}
\left(p^*[\mathcal{O}_{X_u}]\right) \odot\left[\mathcal{O}_{X_{\tau_q}}\right]&=&[ \mathcal{O}_{X_{\overline{u*w_o*\tau_q}}}]\\
&=&[ \mathcal{O}_{X_{\overline{u*\tau_q}}}],\mbox{ by \eqref{5.19.3}.}
\end{eqnarray*}
We claim that 
\begin{equation}\label{5.19.5}
\ell(u\cdot \tau_q) = \ell(u)+\ell(\tau_q)\mbox{ and hence } u*\tau_q =u\cdot \tau_q.
\end{equation}
To prove this, by Lemma \ref{5.17} (c), write 
\begin{equation}\label{5.19.6}
u=x\cdot \tau_{q'},  \mbox{ for some } q'\leq 0 \mbox{ and } x\in W'_{q'}. 
\end{equation}
Thus, 
\begin{eqnarray*}
\ell\left( u\tau_q\right) &=& \ell\left( x\cdot \tau_{q'+q}\right)\\
&=& \ell\left( \tau_{q'+q}\right) -\ell(x),\mbox{ by Lemma \ref{5.17} (b)}\\
&=& \ell\left(\tau_q\right) +\ell\left(\tau_{q'}\right)-\ell(x),\mbox{ by \eqref{5.17.1}}\\
&=& \ell\left(\tau_q\right) +\ell\left(u\right)\mbox{, by Lemma \ref{5.17} (b). }
\end{eqnarray*}
This proves \eqref{5.19.5}. Thus,
\begin{eqnarray*}
\overline{u*\tau_q}= \overline{u\cdot \tau_q} &=&\overline{x\tau_{q'+q}}\mbox{, by \eqref{5.19.6}}\\
&=& x\tau_{q'+q}\mbox{, by Lemma \ref{5.17} (c).}\\
&=& u\cdot\tau_q.
\end{eqnarray*}
Thus, by \eqref{5.19.4},  $\left(p^*[\mathcal{O}_{X_u}]\right) \odot\left[\mathcal{O}_{X_{\tau_q}}\right]=\left[ \mathcal{O}_{X_{{u\cdot\tau_q}}}\right]$, proving \eqref{5.19.2}.
\end{proof}

For any $u,v\in \W'$, recall from the identity \eqref{4.9.1} (identifying $K_0^T(\X)$ as an $R(T)$-submodule of $K_0^T(\Y)$ under 
$p^*: K_0^T(\X) \hookrightarrow K_0^T(\Y)$)
\begin{equation}
\left[\mathcal{O}_{X_u}\right] \odot\left[\mathcal{O}_{X_{v}}\right]=\sum_{w\in \W'} p_{u,v}^w \left[\mathcal{O}_{X_w}\right].\tag{*}
\end{equation}

\begin{corollary}\label{5.20}
For any $q_1,q_2\leq0$ and $u,v,w\in \W'$, 
$$
p_{u,v}^w=p^{w\tau_{q_1+q_2}}_{u\tau_{q_1},v\tau_{q_2}}. 
$$
\end{corollary}

\begin{proof}
Multiply (*) by $\left[\mathcal{O}_{\tau_q}\right]$ and use the associativity and the commutativity of $\odot$ (cf. Corollary \ref{4.11}) and Theorem \ref{5.19}. 
\end{proof}

\section{Example: Convolution product in the case $G=\SL_2$.}

We assume in this section that $G=\SL_2(\mathbbm{C})$ and freely follow the notation from Sections 4 and 5.

\begin{lemma}\label{6.1}
(a) $\bar\zeta^e=e^{-\rho}\otimes e^\rho$, $\bar\zeta^{s_1}=e^0\otimes e^0-e^{-\rho}\otimes e^\rho$ ,  where $\rho=\frac{\alpha_1}{2}$ .
\vskip1ex

(b) $s_0'\left(\bar\zeta^e\right)=e^\rho\otimes e^\rho$, $s_0'\left( \bar\zeta^{s_1}\right)=\bar\zeta^{s_1}+\left(1-e^{2\rho}\right)\bar\zeta^e$ .
\vskip1ex

(c) $D'_0\left(\bar\zeta^e\right)=-e^{2\rho}\bar\zeta^e$, $D'_0\left(\bar\zeta^{s_1}\right)=e^{2\rho}\bar\zeta^e$ .
\vskip1ex

(d) $D'_1\left(\bar\zeta^e\right)=\bar\zeta^e$, $D'_1\left(\bar\zeta^{s_1}\right)=-\bar\zeta^e$ .
\end{lemma}

\begin{proof}
To prove (a), pair the expressions with $\mathcal{O}_{\mathring{X}_e}$ and $\mathcal{O}_{\mathring{X}_{s_1}}$ and use the result $\langle \zeta^x,\mathcal{O}_{\mathring{X}_y}\rangle = \delta_{x, y}$ (cf. [CE, Proposition 3.8]). To prove (b) and (c), recall that $s_0=s_{\alpha_1}=s_1$ and $\alpha_0=-\theta$.
\end{proof}

\begin{remark} \label{6.2} For any simple Lie algebra $\mathfrak{g}$, 
$$\bar{\zeta}^e= e^{-\rho}\otimes e^\rho\,\,\,\text{and}\,\, D'_0(\bar{\zeta}^e)= -(e^\theta+ e^{2 \theta}+ \dots +e^{(h^\vee -1)\theta})\bar{\zeta}^e,$$
where $h^\vee$ is the dual Coxeter number of $\mathfrak{g}$ and $\theta$ is the highest root of $\mathfrak{g}$. 
\end{remark}

For any $n\geq 0$, let 
$
\tau_n:= \ldots s_0s_1s_0$ \,(n-factors).  Then, $\tau_n\in \W'$ and $\tau_{2n}= \tau_{-n\alpha_1^\vee}$ .
Let $X_n :=X_{\tau_n}$. 
 We use Theorem \ref{5.9} and Lemma \ref{6.1} to prove the following. It is obtained in [LLMS, Identity 17] and also in [Ka-1, $\S$2.4] by different methods. (We thank Syu Kato for pointing out these references.)

\begin{proposition} \label{6.3}
For any $n,m\geq0$ , under the (modified) convolution product $\odot$ in $K_0^T(\X)$, 
\vskip1ex
(a) $\left[ \mathcal{O}_n\right] \odot \left[ \mathcal{O}_{2m}\right]=\left[ \mathcal{O}_{n+2m}\right]  $
\vskip1ex
(b) $\left[ \mathcal{O}_{2n+1}\right] \odot \left[ \mathcal{O}_{2m+1}\right]=e^{\alpha_1}\left[ \mathcal{O}_{2n+2m+2}\right]+\left(1-e^{\alpha_1}\right) \left[ \mathcal{O}_{2n+2m+3}\right]$,

where $\mathcal{O}_n$ denotes $\mathcal{O}_{X_n}$. 
\end{proposition} 

\begin{proof}
We first calculate for $v=\tau_{2m}$ (denoting $\mathcal{O}_v=\mathcal{O}_{X_v}$)
\begin{eqnarray*}
\left[ \mathcal{O}_{1}\right] \odot \left[ \mathcal{O}_{2m}\right]&=& \left[ \mathcal{O}_{X_{s_0s_1}^\B}\right] \odot \left[ \mathcal{O}_{\tau_{2m}}\right]\\
&=& \left\vert s_0's_1'\left( \bar{\zeta}^e+\bar{\zeta}^{s_1}\right) \right\vert \left[\mathcal{O}_{s_0s_1*\tau_{2m}}\right]
+ \left\vert D_0's_1' \left(\bar{\zeta}^e+\bar{\zeta}^{s_1}\right) \right\vert \left[\mathcal{O}_{s_1*\tau_{2m}}\right]\\
&&+ \left\vert s_0'D_1' \left(\bar{\zeta}^e+\bar{\zeta}^{s_1}\right) \right\vert \left[\mathcal{O}_{s_0*\tau_{2m}}\right]\\
&&+ \left\vert D_0'D_1' \left(\bar{\zeta}^e+\bar{\zeta}^{s_1}\right) \right\vert \left[\mathcal{O}_{\tau_{2m}}\right],\mbox{ by Theorem \ref{5.9}}\\
&=&  \left[\mathcal{O}_{s_0*\tau_{2m}}\right],\mbox{ by Lemma \ref{6.1}(a), since } \bar{\zeta}^e+\bar{\zeta}^{s_1}= e^0\otimes e^0. 
\end{eqnarray*}
Thus, 
\begin{equation}\label{6.3.1}
\left[ \mathcal{O}_{1}\right] \odot \left[ \mathcal{O}_{2m}\right] = \left[ \mathcal{O}_{2m+1}\right].
\end{equation}
We next calculate  
\begin{eqnarray*}
\left[ \mathcal{O}_{1}\right] \odot \left[ \mathcal{O}_{2m+1}\right] &=& \left[ \mathcal{O}_{X_{s_0s_1}^\B}\right] \odot \left[ \mathcal{O}_{2m+1}\right] \\
&=&  \left\vert D_0'D_1' \left(\bar{\zeta}^e\right) \right\vert \left[\mathcal{O}_{2m+1}\right]+  \left\vert D_0'D_1' \left(\bar{\zeta}^{s_1}\right) \right\vert \left[\mathcal{O}_{s_1*\tau_{2m+1}}\right]\\
&&+ \left\vert D_0's_1' \left(\bar{\zeta}^e+\bar{\zeta}^{s_1}\right) \right\vert \left[\mathcal{O}_{s_1*\tau_{2m+1}}\right]+\left\vert s_0'D_1'(\bar{\zeta}^e)\right\vert \left[\mathcal{O}_{s_0*\tau_{2m+1}}\right]\\
&&+ \left\vert s_0'D_1' (\bar{\zeta}^{s_1}) \right\vert \left[\mathcal{O}_{s_0*s_1*\tau_{2m+1}}\right]\\
&&+\left\vert s_0's_1'\left(\bar{\zeta}^e+\bar{\zeta}^{s_1}\right) \right\vert \left[\mathcal{O}_{s_0*s_1*\tau_{2m+1}}\right],\mbox{ by Theorem \ref{5.9}}\\
&=& -e^{2\rho} \left[\mathcal{O}_{2m+1}\right] +e^{2\rho} \left[\mathcal{O}_{2m+2}\right] + e^{2\rho} \left[\mathcal{O}_{2m+1}\right] \\
&&- e^{2\rho} \left[\mathcal{O}_{2m+3}\right] +\left[\mathcal{O}_{2m+3}\right] \mbox{, by Lemma \ref{6.1} }.
\end{eqnarray*}
Thus, 
\begin{equation}\label{6.3.2}
\left[ \mathcal{O}_{1}\right] \odot \left[ \mathcal{O}_{2m+1}\right] = e^{2\rho} \left[\mathcal{O}_{2m+2}\right] +\left(1-e^{2\rho} \right)\left[\mathcal{O}_{2m+3}\right].
\end{equation}
Similar to the calculation of equation \eqref{6.3.1}, for any $n\geq 0$, writing $\tau_n\cdot s_1=s_{i_1}\ldots s_{i_{n+1}}$ with $s_{i_j}\in \left\{ 0,1\right\}$ and $s_{i_{n+1}}=s_1$ , we get
\begin{eqnarray*}
\left[ \mathcal{O}_{n}\right] \odot \left[ \mathcal{O}_{2m}\right]  &=& \left[ \mathcal{O}_{X_{\tau_n,s_1}^\B}\right] \odot \left[\mathcal{O}_{\tau_{2m}}\right]\\
&=& \sum_{1\leq j_1 <\cdots < j_p \leq n+1}\left\vert D_{i_1}' \circ \cdots \circ \hat{\hat{D}}'_{i_{j_1}}\circ \cdots \circ \hat{\hat{D}}'_{i_{j_p}}\circ \cdots \circ D_{i_{n+1}}'\left(\bar{\zeta}^e+\bar{\zeta}^{s_1}\right) \right\vert \\
&&[ \mathcal{O}_{X_{\overline{s_{i_{j_1}}*\cdots * s_{i_{j_p}}*\tau_{2m}}}}],\mbox{ by Theorem \ref{5.9}} .
\end{eqnarray*}
Thus,
\begin{equation}\label{6.3.3}
\left[ \mathcal{O}_{n}\right] \odot \left[ \mathcal{O}_{2m}\right] =  \left[ \mathcal{O}_{n+2m}\right] ,\mbox{ since } \bar{\zeta}^e+\bar{\zeta}^{s_1} = e^0 \otimes e^0.
\end{equation}
From equation \eqref{6.3.3}, we get (a). 

To prove (b), from (a) we get,
\begin{eqnarray*}
\left[ \mathcal{O}_{2n+1}\right] \odot \left[ \mathcal{O}_{2m+1}\right]  &=& \left(\left[ \mathcal{O}_{1}\right] \odot \left[ \mathcal{O}_{2n}\right] \right)\odot  \left[ \mathcal{O}_{2m+1}\right] \\
&=& \left[ \mathcal{O}_{1}\right] \odot \left[ \mathcal{O}_{2n+2m+1}\right] \mbox{, from the associativity}\\
&&\mbox{ and commutativity of $\odot$ as in  Corollary \ref{4.11}}\\
&=&e^{2\rho}\left[ \mathcal{O}_{2n+2m+2}\right] +\left(1-e^{2\rho}\right)\left[ \mathcal{O}_{2n+2m+3}\right], \mbox{  by \eqref{6.3.2}.}
\end{eqnarray*}
This proves (b).
\end{proof}

\section{Quantum product in equivariant K-theory of flag varieties versus Portryagin product in the loop group}

\begin{definition} \label{7.1}\rm{
Let $Q_+^\vee:= \underset{i=1}{\overset{l}{\oplus}} \mathbbm{Z}_{\geq0}\alpha_i^\vee$, where $\left\{\alpha_1^\vee,\ldots,\alpha_l^\vee\right\}$ are the simple coroots of $G$. Consider the formal power series ring $\mathbbm{Z} [[Q_+^\vee]]$ in the variables $q_i=q^{\alpha_i^\vee}$. For any $\beta=\sum_{i=1}^ln_i\alpha_i^\vee$, $n_i\geq 0$ , we denote $q^\beta=
\prod q_i^{n_i}$.

Additively, \textit{T-equivariant quantum K-theory} of $X=G/B$ is defined as 
$$
Q K_T \left(X\right) = K_T^0 \left(X\right) [[ q_1, \ldots , q_l]].
$$

Thus, $QK_T(X)$ has a $K_T^0(*) [[q_1,\ldots,q_l]] $-basis given by the structure sheaves
$\{ \left[ \mathcal{O}^x\right] = [ \mathcal{O}_{\mathring{X}_{xw_o}}]\}_{x\in W}$, where (as earlier)
$\mathring{X}_{xw_o} \subset X$ is the Schubert variety $\overline{Bxw_oB/B}\subset X$.  
It acquires a ring structure given by Givental [Gi] and Lee [Le]. We denote the product structure by $*$ called the \textit{quantum product}. In this product,  $\left[\mathcal{O}^e\right]\cdot q^0
$ is the identity. Moreover, $\left\{ q^\beta= [\mathcal{O}^e]q^\beta \right\}_{\beta\in Q_+^\vee}
$ forms a multiplicative system. Thus, we can localize $QK_T(X)$ with respect to this multiplicative system to be denoted $QK_T(X)_{\loc}$.

Similarly, by Theorem \ref{5.19} and Lemma \ref{5.17} (b), 
$\left\{ \left[ \mathcal{O}_{X_{\tau_q}}\right]\right\}_{q\in Q_{<0}^\vee}$ forms a multiplicative system in $\left( K_0^T \left(\X\right),\odot \right)$, where 
$$Q_{<0}^\vee := \left\{ q\in Q^\vee:~ \alpha_i (q) <0, \mbox{ for all the simple roots $\alpha_i$ of $G$}\right\} .$$
 Let $K_0^T \left(\X\right)_{\loc}$ denote the localization of $K_0^T \left(\X\right)$ under the modified convolution product $\odot$ with respect to the above multiplicative system. }
\end{definition}

 We recall the following result due to Kato [Ka-1, Corollary 4.21], which was conjectured by [LLMS]. 

\begin{theorem} \label{7.2}There exists an $R(T)$-algebra embedding 
$$
 \psi:  K_0^T \left(\mathcal{X}\right)_{\loc} \hookrightarrow Q  K_T \left(X\right)_{\loc},
$$
such that, for any $\beta,\gamma \in Q_{<0}^\vee$ and $x\in W$ ,
$$
\psi \left( [\mathcal{O}_{X_{x\tau_\beta}}] \odot [ \mathcal{O}_{X_{\tau_\gamma}}]^{-1} \right) = q^{\beta-\gamma}[\mathcal{O}^x]. 
$$
Observe that by Lemma \ref{5.17} (c), $x\tau_{\beta}\in \W'$.
\end{theorem} 

As a corollary of Theorem \ref{7.2}, we get the following. 
\begin{corollary} \label{7.3} For $x,y \in W$ and $\beta_1, \beta_2 \in Q_{<0}^\vee $ , under the quantum product
$$
\left[ \mathcal{O}^x\right] * \left[ \mathcal{O}^y\right] = \sum_{\beta\leq 0, \, z\in W'_\beta} p_{x\tau_{\beta_1},y \tau_{\beta_2}}^{z\tau_\beta}  q^{\beta- \left( \beta_1+\beta_2\right)} \left[\mathcal{O}^z\right] \in QK_T(X),
$$
where $p_{x\tau_{\beta_1},y \tau_{\beta_2}}^{z\tau_\beta}$ are the structure constants for the modified convolution product $\odot$ in $K_0^T \left(\X\right)$ as in  \eqref{4.9.1}. 
\end{corollary}

\begin{proof}
By Theorem 7.2 (abbreviating $\mathcal{O}_{X_u}$ by $\mathcal{O}_u$)
\begin{equation}\label{7.3.1}
\psi \left(( \left[ \mathcal{O}_{x\tau _{\beta_1}}\right] \odot  \left[ \mathcal{O}_{\tau _{\beta_1}}\right]^{-1}) \odot ( \left[ \mathcal{O}_{y\tau _{\beta_2}}\right] \odot \left[\mathcal{O}_{\tau_{\beta_2}}\right]^{-1})\right) 
= \left[ \mathcal{O}^x\right] * \left[ \mathcal{O}^y\right] .
\end{equation}

On the other hand, taking any fixed $\delta<0$, 
\begin{align}\label{7.3.2}
& \psi \left(( \left[ \mathcal{O}_{x\tau _{\beta_1}}\right] \odot  \left[ \mathcal{O}_{\tau _{\beta_1}}\right]^{-1}) \odot ( \left[ \mathcal{O}_{y\tau _{\beta_2}}\right] \odot  \left[ \mathcal{O}_{\tau_{\beta_2}}\right]^{-1}) \right) \nonumber\\
&= \psi \left( \left[ \mathcal{O}_{x\tau _{\beta_1}}\right] \odot  \left[ \mathcal{O}_{y\tau _{\beta_2}}\right] \odot\left[ \mathcal{O}_{\tau _{\beta_1+\beta_2}}\right]^{-1} \right),\mbox{ by Theorem \ref{5.19} and Corollary \ref{4.11}}\nonumber\\
&=  \psi ( \sum_{\beta\leq 0 ,\, z \in W'_\beta} p_{x\tau_{\beta_1},y \tau_{\beta_2}}^{z\tau_\beta} \left[ \mathcal{O}_{z\tau_\beta}\right] \odot 
\left[ \mathcal{O}_{\tau_{\beta_1+\beta_2}}\right]^{-1}),\mbox{ by Lemma \ref{5.17}(c)}\nonumber\\
&= \psi ( \sum_{\beta\leq 0, \,z\in W'_\beta} p_{x\tau_{\beta_1},y \tau_{\beta_2}}^{z\tau_\beta} \left[ \mathcal{O}_{z\tau_\beta}\right] \odot
\left[ \mathcal{O}_{\tau_\delta}\right] \odot \left[ \mathcal{O}_{\tau_\delta}\right] ^{-1}\odot  \left[ \mathcal{O}_{\tau_{\beta+\beta_2}}\right]^{-1})\nonumber\\
&= \psi( \sum_{\beta\leq 0 ,\, z \in W'_\beta} p_{x\tau_{\beta_1},y \tau_{\beta_2}}^{z\tau_\beta}\left[ \mathcal{O}_{z\tau_{\beta+\delta}}\right] \odot
\left[ \mathcal{O}_{\tau_{\delta+\beta_1+\beta_2}}\right]^{-1}),\mbox{ by Theorem \ref{5.19}}\nonumber\\
&= \sum_{\beta\leq 0 , \, z \in W'_\beta} p_{x\tau_{\beta_1},y \tau_{\beta_2}}^{z\tau_\beta} q^{\beta+\delta -\delta - \beta_1-\beta_2}\left[ \mathcal{O}^{z}\right].
\end{align}
Comparing the equations \eqref{7.3.1}  and \eqref{7.3.2}, we get 
$$ 
\left[ \mathcal{O}^x\right] *\left[ \mathcal{O}^y\right] = \sum_{\beta\leq 0 , \, z\in W'_\beta} p_{x\tau_{\beta_1},y \tau_{\beta_2}}^{z\tau_\beta} q^{\beta-\beta_1-\beta_2} \left[ \mathcal{O}^{z}\right].
$$

This proves the Corollary since $QK_T(X) \hookrightarrow QK_T(X)_{\loc}$
(cf. [Ka-1, Proof of Theorem 4.17 and $\S$1.7]).
\end{proof}

\begin{remark}\label{7.4} \rm{
Observe that from the above corollary,  for $\beta_1,\beta_2<0 $ and $\beta\leq 0$ with $z \in W_\beta'$, the structure constants $p_{x\tau_{\beta_1},y \tau_{\beta_2}}^{z\tau_\beta}$ only depend on $x,y, z$ and $\beta-\beta_1-\beta_2$. In particular, 
$$
p_{x\tau_{\beta_1},y \tau_{\beta_2}}^{z\tau_\beta} =p_{x\tau_{\beta_1+\delta_1},y \tau_{\beta_2+\delta_2}}^{z\tau_{\beta+\delta_1+\delta_2}}, \mbox{ for } \delta_1,\delta_2\leq 0.
 $$. 

This is compatible with Corollary \ref{5.20}. }
\end{remark}

\begin{definition}\label{7.5} \rm{
For $x,y\in W$, write the quantum product in $QK_T(X)$: 
$$
 \left[ \mathcal{O}^x\right] *\left[ \mathcal{O}^y\right] = \sum_{z\in W, \, \eta \in Q_+^\vee} d_{x,y}^{z,\eta} q^\eta \left[\mathcal{O}^{z}\right].
$$

As a consequence of Corollary \ref{7.3}, Conjecture \ref{4.10} is equivalent to the following conjecture on the quantum product structure constants in $QK_T(X)$. }
\end{definition}

\begin{conjecture}\label{7.6}
For any $x,y, z \in W$ and $\eta \in Q_+^\vee$, 
$$
(-1)^{\ell(x)+\ell(y)-\ell(z)} d^{z,\eta}_{x,y} \in \mathbbm{Z}_+\left[ \left(e^{\alpha_1}-1\right),\ldots , \left(e^{\alpha_l}-1\right)\right].
$$
\end{conjecture}

\begin{proposition}\label{7.7}
Conjecture \ref{4.10} is equivalent to the above conjecture \ref{7.6}. 
\end{proposition}

\begin{proof}
We first show that Conjecture \ref{4.10}  implies Conjecture \ref{7.6}:

 Fix any $x,y,z \in W$ and $\eta \in Q_+^\vee$. Now, chose any $\beta_1, \beta_2 \in Q_{<0}^\vee$ and $\beta=\eta+\beta_1+\beta_2$ such that $\beta\in Q_{<0}^\vee$. By Corollary \ref{7.3}, 
\begin{equation}\label{7.7.1}
d^{z, \eta}_{x,y} = p_{x\tau_{\beta_1},y \tau_{\beta_2}}^{z\tau_\beta}.
\end{equation}
Observe further that $\ell(\tau_\beta)$ is even for any $\beta\leq 0$ (cf. identity \eqref{5.17.1}). Thus, Conjecture \ref{7.6} follows from that of Conjecture \ref{4.10}. 

Conversely, assume Conjecture \ref{7.6}. Then, for any $x,y,z \in W $ and $\beta,\beta_1,\beta_2 \in Q_{<0}^\vee	$ such that $\beta-\left( \beta_1+ \beta_2\right) \in Q_+^\vee $, we get by \eqref{7.7.1}
\begin{equation}\label{7.7.2}
(-1)^{\ell(x\tau_{\beta_1})+\ell(y\tau_{\beta_2})-\ell(z\tau_\beta)}p_{x\tau_{\beta_1},y \tau_{\beta_2}}^{z\tau_\beta} \in \mathbbm{Z}_+ \left[ (e^{\alpha_1}-1), \ldots ,(e^{\alpha_l}-1)\right].
\end{equation}
Take any $u,v,w\in \W'$ and write (cf. Lemma \ref{5.17}(c)) $u=x\tau_{\gamma_1}$, $v=y\tau_{\gamma_2}$ and $w=z\tau_\gamma$, where $\gamma, \gamma_1, \gamma_2 \leq 0$ and $x\in W'_{\gamma_1}, y\in W'_{\gamma_2}, z\in W'_\gamma$. By Corollary \ref{5.20},
\begin{equation}\label{7.7.3}
p_{u,v}^w = p^{z\tau_{\gamma+\beta_1+\beta_2}}_{x\tau_{\gamma_1+\beta_1},y\tau_{\gamma_2+\beta_2}},\mbox{ for any}\,\, \beta_1,\beta_2 \in  Q_{<0}^\vee.
\end{equation}
By Corollary \ref{7.3}, 
$$
\left[ \mathcal{O}^x\right] *\left[ \mathcal{O}^y\right] =p^{z\tau_{\gamma+\beta_1+\beta_2}}_{x\tau_{\gamma_1+\beta_1},y\tau_{\gamma_2+\beta_2}} q^{\gamma -(\gamma_1+\gamma_2)} \left[ \mathcal{O}^{z}\right] + \mbox{ other terms.}
$$
In particular, if 
\begin{equation}\label{7.7.4}
 p^{z\tau_{\gamma+\beta_1+\beta_2}}_{x\tau_{\gamma_1+\beta_1},y\tau_{\gamma_2+\beta_2}} \neq 0,\quad \mbox{ then }\quad \gamma- (\gamma_1+\gamma_2)\in Q_+^\vee .
\end{equation}

Thus, if non-zero, by the identities \eqref{7.7.3} and \eqref{7.7.4},
\begin{equation}\label{7.7.5}
p_{u,v}^w = p^{z\tau_{\gamma+\beta_1+\beta_2}}_{x\tau_{\gamma_1+\beta_1},y\tau_{\gamma_2+\beta_2}}\quad \mbox{and} \quad \gamma- (\gamma_1+\gamma_2)\in Q_+^\vee. 
\end{equation}
Hence, by the identities \eqref{7.7.2} and \eqref{7.7.5}, 
$$
(-1)^{\ell(u)+\ell(v)-\ell(w)}p_{u,v}^w \in \mathbbm{Z}_+ \left[\left( e^{\alpha_1}-1\right),\ldots, \left(  e^{\alpha_l}-1\right)\right].
$$
 This proves that Conjecture \ref{7.6} implies Conjecture \ref{4.10}. 
\end{proof}

The following example is given in [BM-2, $\S$5.5].
\begin{example} For $G=\SL_2(\mathbb{C})$, we get for $QK_T(\mathbb{P}^1)$ (using Corollary \ref{7.3} and Theorem \ref{5.9}),
$$[\mathcal{O}^{s_1}] * [\mathcal{O}^{s_1}] = e^{\alpha_1}\cdot q^{\alpha^\vee_1} [\mathcal{O}^e]+ (1- e^{\alpha_1})\cdot q^0  [\mathcal{O}^{s_1}].$$
\end{example}

\begin{remark} \label{7.8} \rm{
As proved by Kato [Ka-3], for any standard parabolic subgroup $P$ of $G$, there exists a surjective morphism of commutative $R(T)$-algebras:
$$
QK_T(G/B)\rightarrow QK_T(G/P)
$$
which takes the Schubert basis $\left\{ \left[ \mathcal{O}^x\right] \right\}_{x\in W}$ of $QK_T(G/B)$ to the Schubert basis or zero of $QK_T(G/P)$. Thus, the quantum multiplication structure constants for $G/P$ can be read off from that of $G/B$. }
\end{remark}

\begin{remark}\label{7.9} \rm{
We list some of the known positivity results or conjectures related to $QK(X)$. 
\vskip1ex

(a) Non-equivariant analogue of Conjecture \ref{7.6} for $QK(X)$ (for any $X=G/B$) is made by Lenart-Maeno [LM, Conjecture 7.5]. There is an error in the sign of their conjecture, which they subsequently fixed (as informed to me by C. Lenart). It conforms to our more general $T$-equivariant  conjecture (Conjecture \ref{7.6}). 
\vskip1ex

(b) Buch-Mihalcea [BM-1] conjectured  the $QK_T$ positivity  for Grassmannians. 

\vskip1ex

(c) Lam-Schilling-Shimozono formulated an analogue of Conjecture \ref{2.9} albeit for the structure sheaf basis $\{\mathcal{O}_{X^w}\}_{w\in \W'}$ of $K^{\top} (\X)$ (non-equivariant case) in terms of the multiplicative  structure constants of a basis of the nil-Hecke algebra (cf. [LSS, Conjecture 6.7]). They also have formulated several conjectures on affine stable Grothendieck polynomials and $K$-theoretic $k$-Schur functions (cf. [LSS, Conjectures 7.20 and 7.21]). Parts of their Conjectures 7.20 and 7.21 were subsequently proved by Baldwin-Kumar [BK].

\vskip1ex
(d) Li-Mihalcea [LiM] have proved an  alternating sign behavior for the structure constants associated to line degrees corresponding to some fundamental weights on any $G/P$. 
\vskip1ex

(e) Buch-Chaput-Mihalcea-Perrin [BCMP-1] have proved an analogue of the Chevalley formula with alternating signs for cominuscule flag varieties. They have further proved the non-equivariant analogue of Conjecture \ref{7.6} for minuscule flag varieties as well as quadric hyper surfaces (cf. [BCMP-2]). 
\vskip1ex

(f) Lenart-Naito-Sagaki [LNS] have proved a {\it cancellation free} Chevalley formula with alternating signs for $QK_T(G/B)$. They also have some similar Chevalley formula for Grassmannians in type $A$ and $C$ and some two-step flag manifolds. Also see [BCMP-1] and [KLNS].
\vskip1ex

(g) By a result due to Xu [Xu], Conjecture \ref{7.6} is true for the two-step flag variety of type A. 
\vskip1ex

(h) A positivity result is proved for the symplectic Grassmannian quantum $K$-theory $QK(\IG(2, 2n))$ by Benedetti-Perrin-Xu [BPX]. }

\end{remark}

\vskip5ex

S. Kumar: Department of Mathematics, University of North Carolina, Chapel Hill, NC 27599-3250, USA.

email: shrawan@email.unc.edu

\end{document}